\documentclass{amsart}%
\usepackage{amsmath}
\usepackage{amssymb}
\usepackage{amsfonts}%
\usepackage{graphicx}
\providecommand{\U}[1]{\protect \rule{.1in}{.1in}}
\newtheorem{theorem}{Theorem}
\theoremstyle{plain}

\newtheorem{corollary}{Corollary}

\newtheorem{definition}{Definition}

\newtheorem{lemma}{Lemma}

\newtheorem{remark}{Remark}

\numberwithin{equation}{section}
\begin{document}
\title[some parabolic multilinear commutators]{A note for some parabolic multilinear commutators generated by a class of
parabolic maximal and {linear }operators with rough kernel on the parabolic
generalized local Morrey spaces }
\author{FER\.{I}T G\"{U}RB\"{U}Z}
\address{ANKARA UNIVERSITY, FACULTY OF SCIENCE, DEPARTMENT OF MATHEMATICS, TANDO\u{G}AN
06100, ANKARA, TURKEY }
\curraddr{}
\email{feritgurbuz84@hotmail.com}
\urladdr{}
\thanks{}
\thanks{}
\thanks{}
\date{}
\dedicatory{ }
\begin{abstract}
In this paper, we give the boundedness of some parabolic multilinear
commutators generated by a class of parabolic maximal and linear operators
with rough kernel and parabolic local Campanato functions on the parabolic
generalized local Morrey spaces, respectively. Indeed, the results in this
paper are extensions of some known results.

\end{abstract}
\maketitle

\section{Introduction and main results}

Let $S^{n-1}=\left \{  x\in{\mathbb{R}^{n}:}\text{ }|x|=1\right \}  $ denote the
unit sphere on ${\mathbb{R}^{n}}$ $(n\geq2)$ equipped with the normalized
Lebesgue measure $d\sigma \left(  x^{\prime}\right)  $, where $x^{\prime}$
denotes the unit vector in the direction of $x$ and $\alpha_{n}\geq
\alpha_{n-1}\geq \cdots \geq \alpha_{1}\geq1$ be fixed real numbers.

Note that for each fixed $x=\left(  x_{1},\ldots,x_{n}\right)  \in
{\mathbb{R}^{n}}$, the function%
\[
F\left(  x,\rho \right)  =%
{\displaystyle \sum \limits_{i=1}^{n}}
\frac{x_{i}^{2}}{\rho^{2\alpha_{i}}}%
\]
is a strictly decreasing function of $\rho>0$. Hence, there exists a unique
$\rho=\rho \left(  x\right)  $ such that $F\left(  x,\rho \right)  =1$. It is
clear that for each fixed $x\in{\mathbb{R}^{n}}$, the function $F\left(
x,\rho \right)  $ is a decreasing function in $\rho>0$. Fabes and Rivi\'{e}re
\cite{Fabes and Riviere} showed that $\left(  {\mathbb{R}^{n},}\rho \right)  $
is a metric space which is often called the mixed homogeneity space related to
$\left \{  \alpha_{i}\right \}  _{i=1}^{n}$. For $t>0$, we let $A_{t}$ be the
diagonal $n\times n$ matrix%
\[
A_{t}=diag\left[  t^{\alpha_{1}},\ldots,t^{\alpha_{n}}\right]  =%
\begin{pmatrix}
t^{\alpha_{1}} &  & 0\\
& \ddots & \\
0 &  & t^{\alpha_{n}}%
\end{pmatrix}
.
\]

Let $\rho \in \left(  0,\infty \right)  $ and $0\leq \varphi_{n-1}\leq2\pi$,
$0\leq \varphi_{i}\leq \pi$, $i=1,\ldots,n-2$. For any $x=\left(  x_{1}%
,x_{2},\ldots,x_{n}\right)  \in{\mathbb{R}^{n}}$, set%
\begin{align*}
x_{1}  &  =\rho^{\alpha_{1}}\cos \varphi_{1}\ldots \cos \varphi_{n-2}\cos
\varphi_{n-1},\\
x_{2}  &  =\rho^{\alpha_{2}}\cos \varphi_{1}\ldots \cos \varphi_{n-2}\sin
\varphi_{n-1},\\
&  \vdots \\
x_{n-1}  &  =\rho^{\alpha_{n-1}}\cos \varphi_{1}\sin \varphi_{2},\\
x_{n}  &  =\rho^{\alpha_{n}}\sin \varphi_{1}.
\end{align*}
Thus $dx=\rho^{\alpha-1}J\left(  x^{\prime}\right)  d\rho d\sigma(x^{\prime}%
)$, where $\alpha=%
{\displaystyle \sum \limits_{i=1}^{n}}
\alpha_{i}$, $x^{\prime}\in S^{n-1}$, $J\left(  x^{\prime}\right)  =%
{\displaystyle \sum \limits_{i=1}^{n}}
\alpha_{i}\left(  x_{i}^{\prime}\right)  ^{2}$, $d\sigma$ is the element of
area of $S^{n-1}$ and $\rho^{\alpha-1}J\left(  x^{\prime}\right)  $ is the
Jacobian of the above transform. Obviously, $J\left(  x^{\prime}\right)  \in$
$C^{\infty}\left(  S^{n-1}\right)  $ function and that there exists $M>0$ such
that $1\leq J\left(  x^{\prime}\right)  \leq M$ and $x^{\prime}\in S^{n-1}$.

Let $P$ be a real $n\times n$ matrix, whose all the eigenvalues have positive
real part. Let $A_{t}=t^{P}$ $\left(  t>0\right)  $, and set $\gamma=trP$.
Then, there exists a quasi-distance $\rho$ associated with $P$ such that (see
\cite{Coifman-Weiss})

$\left(  1-1\right)  $ $\rho \left(  A_{t}x\right)  =t\rho \left(  x\right)  $,
$t>0$, for every $x\in{\mathbb{R}^{n}}$,

$\left(  1-2\right)  $ $\rho \left(  0\right)  =0$, $\rho \left(  x-y\right)
=\rho \left(  y-x\right)  \geq0$, and $\rho \left(  x-y\right)  \leq k\left(
\rho \left(  x-z\right)  +\rho \left(  y-z\right)  \right)  $,

$\left(  1-3\right)  $ $dx=\rho^{\gamma-1}d\sigma \left(  w\right)  d\rho$,
where $\rho=\rho \left(  x\right)  $, $w=A_{\rho^{-1}}x$ and $d\sigma \left(
w\right)  $ is a measure on the unit ellipsoid $\left \{  w:\rho \left(
w\right)  =1\right \}  $.

Then, $\left \{  {\mathbb{R}^{n},\rho,dx}\right \}  $ becomes a space of
homogeneous type in the sense of Coifman-Weiss (see \cite{Coifman-Weiss}) and
a homogeneous group in the sense of Folland-Stein (see \cite{Folland-Stein}).

Denote by $E\left(  x,r\right)  $ the ellipsoid with center at $x$ and radius
$r$, more precisely, $E\left(  x,r\right)  =\left \{  y\in{\mathbb{R}^{n}:\rho
}\left(  x-y\right)  <r\right \}  $. For $k>0$, we denote $kE\left(
x,r\right)  =\left \{  y\in{\mathbb{R}^{n}:\rho}\left(  x-y\right)
<kr\right \}  $. Moreover, by the property of $\rho$ and the polar coordinates
transform above, we have%
\[
\left \vert E\left(  x,r\right)  \right \vert =%
{\displaystyle \int \limits_{{\rho}\left(  x-y\right)  <r}}
dy=\upsilon_{\rho}r^{\alpha_{1}+\cdots+\alpha_{n}}=\upsilon_{\rho}r^{\gamma},
\]
where $|E(x,r)|$ stands for the Lebesgue measure of $E(x,r)$ and
$\upsilon_{\rho}$ is the volume of the unit ellipsoid on ${\mathbb{R}^{n}}$.
By $E^{C}(x,r)={\mathbb{R}^{n}}\setminus$ $E\left(  x,r\right)  $, we denote
the complement of $E\left(  x,r\right)  $. If we take $\alpha_{1}%
=\cdots=\alpha_{n}=1$ and $P=I$, then obviously $\rho \left(  x\right)
=\left \vert x\right \vert =\left(
{\displaystyle \sum \limits_{i=1}^{n}}
x_{i}^{2}\right)  ^{\frac{1}{2}} $, $\gamma=n$, $\left(  {\mathbb{R}^{n},\rho
}\right)  =$ $\left(  {\mathbb{R}^{n},}\left \vert \cdot \right \vert \right)  $,
$E_{I}(x,r)=B\left(  x,r\right)  $, $A_{t}=tI$ and $J\left(  x^{\prime
}\right)  \equiv1$. Moreover, in the standard parabolic case $P_{0}%
=diag\left[  1,\ldots,1,2\right]  $ we have%
\[
\rho \left(  x\right)  =\sqrt{\frac{\left \vert x^{\prime}\right \vert ^{2}%
+\sqrt{\left \vert x^{\prime}\right \vert ^{4}+x_{n}^{2}}}{2},}\qquad x=\left(
x^{\prime},x_{n}\right)  .
\]

Note that we deal not exactly with the parabolic metric, but with a general
anisotropic metric $\rho$ of generalized homogeneity, the parabolic metric
being its particular case, but we keep the term parabolic in the title and
text of the paper, the above existing tradition, see for instance
\cite{Calderon and Torchinsky}.

Suppose that $\Omega \left(  x\right)  $ is a real-valued and measurable
function defined on ${\mathbb{R}^{n}}$. Suppose that $S^{n-1}$ is the unit
sphere on ${\mathbb{R}^{n}}$ $(n\geq2)$ equipped with the normalized Lebesgue
surface measure $d\sigma$. Let $\Omega \in L_{s}(S^{n-1})$ with $1<s\leq \infty$
be homogeneous of degree zero with respect to $A_{t}$ ($\Omega \left(
x\right)  $ is $A_{t}$-homogeneous of degree zero), that is, $\Omega
(A_{t}x)=\Omega(x),~~$for any$~~t>0,$ $x\in{\mathbb{R}^{n}}$. We define
$s^{\prime}=\frac{s}{s-1}$ for any $s>1$. One of the important problems on
parabolic homogeneous spaces investigates the boundedness of parabolic linear
operators satisfying the following size conditions ((\ref{e1}) and
(\ref{e2})). Therefore, in this paper, we consider parabolic linear operators
$T_{\Omega}^{P}$ and $T_{\Omega,\alpha}^{P}$, $\alpha \in \left(  0,\gamma
\right)  $ satisfying the size conditions for any $f\in L_{1}({\mathbb{R}^{n}%
})$ with compact support and $x\notin suppf$, respectively%
\begin{equation}
|T_{\Omega}^{P}f(x)|\leq c_{0}\int \limits_{{\mathbb{R}^{n}}}\frac
{|\Omega(x-y)|}{{\rho}\left(  x-y\right)  ^{\gamma}}\,|f(y)|\,dy,\label{e1}%
\end{equation}%
\begin{equation}
|T_{\Omega,\alpha}^{P}f(x)|\leq c_{0}\int \limits_{{\mathbb{R}^{n}}}%
\frac{|\Omega(x-y)|}{{\rho}\left(  x-y\right)  ^{\gamma-\alpha}}%
\,|f(y)|\,dy,\label{e2}%
\end{equation}
where $c_{0}$ is independent of $f$ and $x$.

We point out that the conditions (\ref{e1}) and (\ref{e2}) in the case
$\Omega \equiv1$, $\alpha=0$ and $P=I$ was first introduced by Soria and Weiss
in \cite{SW}. Indeed, in 1944, Soria and Weiss developed Stein's result
\cite{St} in the above shape. The conditions (\ref{e1}) and (\ref{e2}) are
satisfied by many interesting operators in harmonic analysis, such as the
parabolic Calder\'{o}n--Zygmund operators, parabolic Carleson's maximal
operator, parabolic Hardy--Littlewood maximal operator, parabolic C.
Fefferman's singular multipliers, parabolic R. Fefferman's singular integrals,
parabolic Ricci--Stein's oscillatory singular integrals, parabolic the
Bochner--Riesz means, the parabolic fractional integral operator(parabolic
Riesz potential), parabolic fractional maximal operator, parabolic fractional
Marcinkiewicz operator and so on (see \cite{Gurbuz2, Gurbuz3, SW} for details).

The parabolic fractional maximal function $M_{\Omega,\alpha}^{P}f$ and
$T_{\Omega,\alpha}^{p}f$ by with rough kernels, $0<\alpha<\gamma$, of a
function $f\in L^{loc}\left(  {\mathbb{R}^{n}}\right)  $ are defined by%
\[
M_{\Omega,\alpha}^{P}f(x)=\sup_{t>0}|E(x,t)|^{-1+\frac{\alpha}{\gamma}}%
\int \limits_{E(x,t)}\left \vert \Omega \left(  x-y\right)  \right \vert |f(y)|dy,
\]%
\[
T_{\Omega,\alpha}^{P}f(x)=\int \limits_{{\mathbb{R}^{n}}}\frac{\Omega
(x-y)}{{\rho}\left(  x-y\right)  ^{\gamma-\alpha}}f(y)dy,
\]
satisfy condition (\ref{e2}). It is obvious that when $\Omega \equiv1$,
$M_{1,\alpha}^{P}\equiv M_{\alpha}^{p}$ and $T_{1,\alpha}^{P}\equiv T_{\alpha
}^{P}$ are the parabolic fractional maximal operator and the parabolic
fractional integral operator, respectively. If $P=I$, then $M_{\Omega,\alpha
}^{I}\equiv M_{\Omega,\alpha}$ and $T_{\Omega,\alpha}^{I}\equiv T_{\Omega
,\alpha}$ are the fractional maximal operator with rough kernel and fractional
integral operator with rough kernel, respectively. It is well known that the
parabolic fractional maximal and integral operators play an important role in
harmonic analysis (see \cite{Calderon and Torchinsky, Folland-Stein, Gurbuz3}).

We notice that when $\alpha=0$, the above operators become the parabolic
Calder\'{o}n--Zygmund singular integral operator with rough kernel $T_{\Omega
}^{P}=T_{\Omega,0}^{P}$ and the corresponding parabolic maximal operator with
rough kernel $M_{\Omega,0}^{P}\equiv M_{\Omega}^{P}$:%

\[
T_{\Omega}^{P}f(x)=p.v.\int \limits_{{\mathbb{R}^{n}}}\frac{\Omega(x-y)}{{\rho
}\left(  x-y\right)  ^{\gamma}}f(y)dy,
\]%
\[
M_{\Omega}^{P}f(x)=\sup_{t>0}|E(x,t)|^{-1}\int \limits_{E(x,t)}\left \vert
\Omega \left(  x-y\right)  \right \vert |f(y)|dy,
\]
satisfy condition (\ref{e1}). It is obvious that when $\Omega \equiv1$,
$T_{\Omega}^{P}\equiv T^{P}$ and $M_{\Omega}^{P}\equiv M^{P}$ are the
parabolic singular operator and the parabolic maximal operator, respectively.
If $P=I$, then $M_{\Omega}^{I}\equiv M_{\Omega}$ is the Hardy-Littlewood
maximal operator with rough kernel, and $T_{\Omega}^{I}\equiv T_{\Omega}$ is
the homogeneous singular integral operator. It is well known that the
parabolic maximal and singular operators play an important role in harmonic
analysis (see \cite{Calderon and Torchinsky, Folland-Stein, Gurbuz2, Xue}).

On the other hand let $b$ be a locally integrable function on ${\mathbb{R}%
^{n}}$, then for $0<\alpha<\gamma$, we define commutators generated by
parabolic fractional maximal and integral operators with rough kernel and $b$
as follows, respectively.%
\[
M_{\Omega,b,\alpha}^{P}\left(  f\right)  (x)=\sup_{t>0}|E(x,t)|^{-1+\frac
{\alpha}{\gamma}}\int \limits_{E(x,t)}\left \vert b\left(  x\right)  -b\left(
y\right)  \right \vert \left \vert \Omega \left(  x-y\right)  \right \vert
|f(y)|dy,
\]%
\[
\lbrack b,T_{\Omega,\alpha}^{P}]f(x)\equiv b(x)T_{\Omega,\alpha}%
^{P}f(x)-T_{\Omega,\alpha}^{P}(bf)(x)=\int \limits_{{\mathbb{R}^{n}}%
}[b(x)-b(y)]\frac{\Omega(x-y)}{{\rho}\left(  x-y\right)  ^{\gamma-\alpha}%
}f(y)dy.
\]

Similarly, for $\alpha=0$, we define commutators generated by parabolic
maximal and singular integral operators by with rough kernels and $b$ as
follows, respectively.%
\[
M_{\Omega,b}^{P}\left(  f\right)  (x)=\sup_{t>0}|E(x,t)|^{-1}\int
\limits_{E(x,t)}\left \vert b\left(  x\right)  -b\left(  y\right)  \right \vert
\left \vert \Omega \left(  x-y\right)  \right \vert |f(y)|dy,
\]%
\[
\lbrack b,T_{\Omega}^{P}]f(x)\equiv b(x)T_{\Omega}^{P}f(x)-T_{\Omega}%
^{P}(bf)(x)=p.v.\int \limits_{{\mathbb{R}^{n}}}[b(x)-b(y)]\frac{\Omega
(x-y)}{{\rho}\left(  x-y\right)  ^{\gamma}}f(y)dy.
\]

Because of the need for the study of partial differential equations (PDEs),
Morrey \cite{Morrey} introduced Morrey spaces $M_{p,\lambda}$ which naturally
are generalizations of Lebesgue spaces. We also refer to \cite{Adams} for the
latest research on the theory of Morrey spaces associated with harmonic analysis.

A measurable function $f\in L_{p}\left(  {\mathbb{R}^{n}}\right)  $,
$p\in \left(  1,\infty \right)  $, belongs to the parabolic Morrey spaces
$M_{p,\lambda,P}\left(  {\mathbb{R}^{n}}\right)  $ with $\lambda \in \left[
0,\gamma \right)  $ if the following norm is finite:%

\[
\left \Vert f\right \Vert _{M_{p,\lambda,P}}=\left(  \sup_{x\in{\mathbb{R}%
^{n},r>0}}\frac{1}{r^{\lambda}}%
{\displaystyle \int \limits_{E(x,r)}}
\left \vert f\left(  y\right)  \right \vert ^{p}dy\right)  ^{1/p},
\]
where $E(x,r)$ stands for any ellipsoid with center at $x$ and radius $r$.
When $\lambda=0$, $M_{p,\lambda,P}\left(  {\mathbb{R}^{n}}\right)  $ coincides
with the parabolic Lebesgue space $L_{p,P}\left(  {\mathbb{R}^{n}}\right)  $.

If $P=I$, then $M_{p,\lambda,I}({\mathbb{R}^{n}})\equiv M_{p,\lambda
}({\mathbb{R}^{n}})$ and $L_{p,I}\left(  {\mathbb{R}^{n}}\right)  \equiv
L_{p}\left(  {\mathbb{R}^{n}}\right)  $ are the classical Morrey and the
Lebesgue spaces, respectively.

We now recall the definition of parabolic generalized local (central) Morrey
space $LM_{p,\varphi,P}^{\{x_{0}\}}$ in the following.

\begin{definition}
\label{Definition2}\cite{Gurbuz2, Gurbuz3} \textbf{(parabolic generalized
local (central) Morrey space) }Let $\varphi(x,r)$ be a positive measurable
function on ${\mathbb{R}^{n}}\times(0,\infty)$ and $1\leq p<\infty$. For any
fixed $x_{0}\in{\mathbb{R}^{n}}$ we denote by $LM_{p,\varphi,P}^{\{x_{0}%
\}}\equiv LM_{p,\varphi,P}^{\{x_{0}\}}({\mathbb{R}^{n}})$ the parabolic
generalized local Morrey space, the space of all functions $f\in L_{p}%
^{loc}({\mathbb{R}^{n}})$ with finite quasinorm
\[
\Vert f\Vert_{LM_{p,\varphi,P}^{\{x_{0}\}}}=\sup \limits_{r>0}\varphi
(x_{0},r)^{-1}|E(x_{0},r)|^{-\frac{1}{p}}\Vert f\Vert_{L_{p}(E(x_{0}%
,r))}<\infty.
\]
According to this definition, we recover the local parabolic Morrey space
$LM_{p,\lambda,P}^{\{x_{0}\}}$ and weak local parabolic Morrey space
$WLM_{p,\lambda,P}^{\{x_{0}\}}$ under the choice $\varphi(x_{0},r)=r^{\frac
{\lambda-\gamma}{p}}$:%
\[
LM_{p,\lambda,P}^{\{x_{0}\}}=LM_{p,\varphi,P}^{\{x_{0}\}}\mid_{\varphi
(x_{0},r)=r^{\frac{\lambda-\gamma}{p}}},~~~~~~WLM_{p,\lambda,P}^{\{x_{0}%
\}}=WLM_{p,\varphi,P}^{\{x_{0}\}}\mid_{\varphi(x_{0},r)=r^{\frac
{\lambda-\gamma}{p}}}.
\]

\end{definition}

Now, let us recall the defination of the space of $LC_{p,\lambda,P}^{\left \{
x_{0}\right \}  }$ (parabolic local Campanato space{)}.

\begin{definition}
\cite{Gurbuz2, Gurbuz3} Let $1\leq p<\infty$ and $0\leq \lambda<\frac{1}%
{\gamma}$. A parabolic local Campanato function $b\in L_{p}^{loc}\left(
{\mathbb{R}^{n}}\right)  $ is said to belong to the $LC_{p,\lambda
,P}^{\left \{  x_{0}\right \}  }\left(  {\mathbb{R}^{n}}\right)  $, if%
\[
\left \Vert b\right \Vert _{LC_{p,\lambda,P}^{\left \{  x_{0}\right \}  }}%
=\sup_{r>0}\left(  \frac{1}{\left \vert E\left(  x_{0},r\right)  \right \vert
^{1+\lambda p}}\int \limits_{E\left(  x_{0},r\right)  }\left \vert b\left(
y\right)  -b_{E\left(  x_{0},r\right)  }\right \vert ^{p}dy\right)  ^{\frac
{1}{p}}<\infty,
\]
where%
\[
b_{E\left(  x_{0},r\right)  }=\frac{1}{\left \vert E\left(  x_{0},r\right)
\right \vert }\int \limits_{E\left(  x_{0},r\right)  }b\left(  y\right)  dy.
\]

Define%
\[
LC_{p,\lambda,P}^{\left \{  x_{0}\right \}  }\left(  {\mathbb{R}^{n}}\right)
=\left \{  b\in L_{p}^{loc}\left(  {\mathbb{R}^{n}}\right)  :\left \Vert
b\right \Vert _{LC_{p,\lambda,P}^{\left \{  x_{0}\right \}  }}<\infty \right \}  .
\]

\end{definition}

Let $b_{i}\left(  i=1,\ldots,m\right)  $ be locally integrable functions on
${\mathbb{R}^{n}}$, then the fractional type parabolic multilinear commutators
generated by parabolic fractional maximal and integral operators with rough
kernel and $\overrightarrow{b}=\left(  b_{1},\ldots,b_{m}\right)  $ (parabolic
local Campanato functions{) }are given as follows, respectively:%
\[
\lbrack \overrightarrow{b},T_{\Omega,\alpha}^{P}]f\left(  x\right)  =%
{\displaystyle \int \limits_{{\mathbb{R}^{n}}}}
{\displaystyle \prod \limits_{i=1}^{m}}
\left[  b_{i}\left(  x\right)  -b_{i}\left(  y\right)  \right]  \frac
{\Omega(x-y)}{{\rho}\left(  x-y\right)  ^{\gamma-\alpha}}f\left(  y\right)
dy,\qquad0<\alpha<\gamma,
\]%
\[
M_{\Omega,\overrightarrow{b},\alpha}^{P}f\left(  x\right)  =\sup
_{t>0}|E(x,t)|^{-1+\frac{\alpha}{\gamma}}%
{\displaystyle \int \limits_{E(x,t)}}
{\displaystyle \prod \limits_{i=1}^{m}}
\left[  \left \vert b_{i}\left(  x\right)  -b_{i}\left(  y\right)  \right \vert
\right]  \left \vert \Omega(x-y)\right \vert |f(y)|dy,\qquad0<\alpha<\gamma.
\]
We notice that when $\alpha=0$, the above operators become the parabolic
multilinear commutators generated by parabolic singular integral operators and
the corresponding parabolic\ maximal operators with rough kernel and
$\overrightarrow{b}=\left(  b_{1},\ldots,b_{m}\right)  ${\ }as follows,
respectively:%
\[
\lbrack \overrightarrow{b},T_{\Omega}^{P}]f\left(  x\right)  =%
{\displaystyle \int \limits_{{\mathbb{R}^{n}}}}
{\displaystyle \prod \limits_{i=1}^{m}}
\left[  b_{i}\left(  x\right)  -b_{i}\left(  y\right)  \right]  \frac
{\Omega(x-y)}{{\rho}\left(  x-y\right)  ^{\gamma}}f\left(  y\right)  dy,
\]%
\[
M_{\Omega,\overrightarrow{b}}^{P}f\left(  x\right)  =\sup_{t>0}|E(x,t)|^{-1}%
{\displaystyle \int \limits_{E(x,t)}}
{\displaystyle \prod \limits_{i=1}^{m}}
\left[  \left \vert b_{i}\left(  x\right)  -b_{i}\left(  y\right)  \right \vert
\right]  \left \vert \Omega(x-y)\right \vert |f(y)|dy.
\]

In \cite{Gurbuz2, Gurbuz3} the boundedness of a class of parabolic sublinear
operators with rough kernel and their commutators on the parabolic generalized
local Morrey spaces under generic size conditions which are satisfied by most
of the operators in harmonic analysis has been investigated, respectively.

Inspired by \cite{Gurbuz2, Gurbuz3}, our main purpose in this paper is to
consider the boundedness of above operators ($[\overrightarrow{b},T_{\Omega}%
]$, $M_{\Omega,\overrightarrow{b}}$, $[\overrightarrow{b},T_{\Omega,\alpha}]$,
$M_{\Omega,\overrightarrow{b},\alpha}$) on the parabolic generalized local
Morrey spaces, respectively. But, the techniques and non-trivial estimates
which have been used in the proofs of our main results are quite different
from \cite{Gurbuz2, Gurbuz3}. For example, using inequality about the weighted
Hardy operator $H_{w}$ in \cite{Gurbuz2, Gurbuz3}, in this paper we will only
use the following relationship between essential supremum and essential
infimum%
\begin{equation}
\left(  \operatorname*{essinf}\limits_{x\in E}f\left(  x\right)  \right)
^{-1}=\operatorname*{esssup}\limits_{x\in E}\frac{1}{f\left(  x\right)
},\label{5}%
\end{equation}
where $f$ is any real-valued nonnegative function and measurable on $E$ (see
\cite{Wheeden-Zygmund}, page 143). Our main results can be formulated as follows.

\begin{theorem}
\label{teo4}Suppose that $x_{0}\in{\mathbb{R}^{n}}$, $\Omega \in L_{s}%
(S^{n-1})$, $1<s\leq \infty$, is $A_{t}$-homogeneous of degree zero. Let
$T_{\Omega}^{P}$ be a parabolic linear operator satisfying condition
(\ref{e1}). Let also $1<q,p_{i},p<\infty$ with $\frac{1}{q}=\sum
\limits_{i=1}^{m}\frac{1}{p_{i}}+\frac{1}{p}$ and $\overrightarrow{b}\in
LC_{p_{i},\lambda_{i},P}^{\left \{  x_{0}\right \}  }({\mathbb{R}^{n}})$ for
$0\leq \lambda_{i}<\frac{1}{\gamma}$, $i=1,\ldots,m$.

Let also, for $s^{\prime}\leq q$ the pair $(\varphi_{1},\varphi_{2})$
satisfies the condition%
\begin{equation}%
{\displaystyle \int \limits_{r}^{\infty}}
\left(  1+\ln \frac{t}{r}\right)  ^{m}\frac{\operatorname*{essinf}%
\limits_{t<\tau<\infty}\varphi_{1}(x_{0},\tau)\tau^{\frac{\gamma}{p}}%
}{t^{\gamma \left(  \frac{1}{p}-%
{\displaystyle \sum \limits_{i=1}^{m}}
\lambda_{i}\right)  +1}}\leq C\, \varphi_{2}(x_{0},r),\label{47}%
\end{equation}
and for $p<s$ the pair $(\varphi_{1},\varphi_{2})$ satisfies the condition%
\[
\int \limits_{r}^{\infty}\left(  1+\ln \frac{t}{r}\right)  ^{m}\frac
{\operatorname*{essinf}\limits_{t<\tau<\infty}\varphi_{1}(x_{0},\tau
)\tau^{\frac{\gamma}{p}}}{t^{\gamma \left(  \frac{1}{p}-\frac{1}{s}-%
{\displaystyle \sum \limits_{i=1}^{m}}
\lambda_{i}\right)  +1}}dt\leq C\, \varphi_{2}(x_{0},r)r^{\frac{\gamma}{s}},
\]
where $C$ does not depend on $r$.

Then, the operators $[\overrightarrow{b},T_{\Omega}^{P}]$ and $M_{\Omega
,\overrightarrow{b}}^{P}$ are bounded from $LM_{p,\varphi_{1},P}^{\{x_{0}\}}$
to $LM_{q,\varphi_{2},P}^{\{x_{0}\}}$. Moreover,
\begin{equation}
\left \Vert \lbrack \overrightarrow{b},T_{\Omega}^{P}]f\right \Vert
_{LM_{q,\varphi_{2},P}^{\{x_{0}\}}}\lesssim%
{\displaystyle \prod \limits_{i=1}^{m}}
\Vert \overrightarrow{b}\Vert_{LC_{p_{i},\lambda_{i},P}^{\left \{
x_{0}\right \}  }}\left \Vert f\right \Vert _{LM_{p,\varphi_{1},P}^{\{x_{0}\}}%
},\label{13}%
\end{equation}%
\begin{equation}
\left \Vert \lbrack \overrightarrow{b},M_{\Omega,\overrightarrow{b}}%
^{P}]f\right \Vert _{LM_{q,\varphi_{2},P}^{\{x_{0}\}}}\lesssim%
{\displaystyle \prod \limits_{i=1}^{m}}
\Vert \overrightarrow{b}\Vert_{LC_{p_{i},\lambda_{i},P}^{\left \{
x_{0}\right \}  }}\left \Vert f\right \Vert _{LM_{p,\varphi_{1},P}^{\{x_{0}\}}%
}.\label{14}%
\end{equation}

\end{theorem}

\begin{corollary}
\cite{Gurbuz2} Suppose that $x_{0}\in{\mathbb{R}^{n}}$, $\Omega \in
L_{s}(S^{n-1})$, $1<s\leq \infty$, is $A_{t}$-homogeneous of degree zero. Let
$T_{\Omega}^{P}$ be a parabolic linear operator satisfying condition
(\ref{e1}), bounded on $L_{p}({\mathbb{R}^{n}})$ for $1<p<\infty$. Let $b\in
LC_{p_{2},\lambda,P}^{\left \{  x_{0}\right \}  }\left(
\mathbb{R}
^{n}\right)  $, $0\leq \lambda<\frac{1}{\gamma}$ and $\frac{1}{p}=\frac
{1}{p_{1}}+\frac{1}{p_{2}}$. Let also, for $s^{\prime}\leq p$ the pair
$(\varphi_{1},\varphi_{2})$ satisfies the condition%
\[
\int \limits_{r}^{\infty}\left(  1+\ln \frac{t}{r}\right)  \frac
{\operatorname*{essinf}\limits_{t<\tau<\infty}\varphi_{1}(x_{0},\tau
)\tau^{\frac{\gamma}{p_{1}}}}{t^{\frac{\gamma}{p_{1}}+1-\gamma \lambda}}dt\leq
C\, \varphi_{2}(x_{0},r),
\]
and for $p_{1}<s$ the pair $(\varphi_{1},\varphi_{2})$ satisfies the condition%
\[
\int \limits_{r}^{\infty}\left(  1+\ln \frac{t}{r}\right)  \frac
{\operatorname*{essinf}\limits_{t<\tau<\infty}\varphi_{1}(x_{0},\tau
)\tau^{\frac{\gamma}{p_{1}}}}{t^{\frac{\gamma}{p_{1}}-\frac{\gamma}%
{s}+1-\gamma \lambda}}dt\leq C\, \varphi_{2}(x_{0},r)r^{\frac{\gamma}{s}},
\]
where $C$ does not depend on $r$.

Then, the operators $[b,T_{\Omega}^{P}]$ and $M_{\Omega,b}^{P}$ are bounded
from $LM_{p_{1},\varphi_{1},P}^{\{x_{0}\}}$ to $LM_{p,\varphi_{2},P}%
^{\{x_{0}\}}$. Moreover,%
\[
\left \Vert \lbrack b,T_{\Omega}^{P}]f\right \Vert _{LM_{p,\varphi_{2}%
,P}^{\{x_{0}\}}}\lesssim \left \Vert b\right \Vert _{LC_{p_{2},\lambda
,P}^{\left \{  x_{0}\right \}  }}\left \Vert f\right \Vert _{LM_{p_{1},\varphi
_{1},P}^{\{x_{0}\}}},
\]%
\[
\left \Vert M_{\Omega,b}^{P}f\right \Vert _{LM_{p,\varphi_{2},P}^{\{x_{0}\}}%
}\lesssim \left \Vert b\right \Vert _{LC_{p_{2},\lambda,P}^{\left \{
x_{0}\right \}  }}\left \Vert f\right \Vert _{LM_{p_{1},\varphi_{1},P}%
^{\{x_{0}\}}}.
\]

\end{corollary}

\begin{theorem}
\label{teo4*}Suppose that $x_{0}\in{\mathbb{R}^{n}}$, $\Omega \in L_{s}%
(S^{n-1})$, $1<s\leq \infty$, is $A_{t}$-homogeneous of degree zero. Let
$T_{\Omega,\alpha}^{P}$ be a parabolic linear operator satisfying condition
(\ref{e2}). Let also $0<\alpha<\gamma$ and $1<q,q_{1},p_{i},p<\frac{\gamma
}{\alpha}$ with $\frac{1}{q}=\sum \limits_{i=1}^{m}\frac{1}{p_{i}}+\frac{1}{p}%
$, $\frac{1}{q_{1}}=\frac{1}{q}-\frac{\alpha}{\gamma}$ and $\overrightarrow
{b}\in LC_{p_{i},\lambda_{i},P}^{\left \{  x_{0}\right \}  }({\mathbb{R}^{n}})$
for $0\leq \lambda_{i}<\frac{1}{\gamma}$, $i=1,\ldots,m$.

Let also, for $s^{\prime}\leq q$ the pair $(\varphi_{1},\varphi_{2})$
satisfies the condition%
\begin{equation}%
{\displaystyle \int \limits_{r}^{\infty}}
\left(  1+\ln \frac{t}{r}\right)  ^{m}\frac{\operatorname*{essinf}%
\limits_{t<\tau<\infty}\varphi_{1}(x_{0},\tau)\tau^{\frac{\gamma}{p}}%
}{t^{\gamma \left(  \frac{1}{q_{1}}-\left(
{\displaystyle \sum \limits_{i=1}^{m}}
\lambda_{i}+%
{\displaystyle \sum \limits_{i=1}^{m}}
\frac{1}{p_{i}}\right)  \right)  +1}}\leq C\, \varphi_{2}(x_{0},r),\label{47*}%
\end{equation}
and for $q_{1}<s$ the pair $(\varphi_{1},\varphi_{2})$ satisfies the condition%
\[
\int \limits_{r}^{\infty}\left(  1+\ln \frac{t}{r}\right)  ^{m}\frac
{\operatorname*{essinf}\limits_{t<\tau<\infty}\varphi_{1}(x_{0},\tau
)\tau^{\frac{\gamma}{p}}}{t^{\gamma \left(  \frac{1}{q_{1}}-\left(  \frac{1}%
{s}+%
{\displaystyle \sum \limits_{i=1}^{m}}
\lambda_{i}+%
{\displaystyle \sum \limits_{i=1}^{m}}
\frac{1}{p_{i}}\right)  \right)  +1}}dt\leq C\, \varphi_{2}(x_{0}%
,r)r^{\frac{\gamma}{s}},
\]
where $C$ does not depend on $r$.

Then, the operators $[\overrightarrow{b},T_{\Omega,\alpha}^{P}]$ and
$M_{\Omega,\overrightarrow{b},\alpha}^{P}$ are bounded from $LM_{p,\varphi
_{1},P}^{\{x_{0}\}}$ to $LM_{q_{1},\varphi_{2},P}^{\{x_{0}\}}$. Moreover,
\begin{equation}
\left \Vert \lbrack \overrightarrow{b},T_{\Omega,\alpha}^{P}]f\right \Vert
_{LM_{q_{1},\varphi_{2},P}^{\{x_{0}\}}}\lesssim%
{\displaystyle \prod \limits_{i=1}^{m}}
\Vert \overrightarrow{b}\Vert_{LC_{p_{i},\lambda_{i},P}^{\left \{
x_{0}\right \}  }}\left \Vert f\right \Vert _{LM_{p,\varphi_{1},P}^{\{x_{0}\}}%
},\label{13*}%
\end{equation}%
\begin{equation}
\left \Vert M_{\Omega,\overrightarrow{b},\alpha}^{P}f\right \Vert _{LM_{q_{1}%
,\varphi_{2},P}^{\{x_{0}\}}}\lesssim%
{\displaystyle \prod \limits_{i=1}^{m}}
\Vert \overrightarrow{b}\Vert_{LC_{p_{i},\lambda_{i},P}^{\left \{
x_{0}\right \}  }}\left \Vert f\right \Vert _{LM_{p,\varphi_{1},P}^{\{x_{0}\}}%
}.\label{14*}%
\end{equation}

\end{theorem}

\begin{corollary}
\cite{Gurbuz3} Suppose that $x_{0}\in{\mathbb{R}^{n}}$, $\Omega \in
L_{s}(S^{n-1})$, $1<s\leq \infty$, is $A_{t}$-homogeneous of degree zero. Let
$T_{\Omega,\alpha}^{P}$ be a parabolic linear operator satisfying condition
(\ref{e2}) and bounded from $L_{p}({\mathbb{R}^{n}})$ to $L_{q}({\mathbb{R}%
^{n}})$. Let $0<\alpha<\gamma$, $1<p<\frac{\gamma}{\alpha}$, $b\in
LC_{p_{2},\lambda,P}^{\left \{  x_{0}\right \}  }\left(
\mathbb{R}
^{n}\right)  $, $0\leq \lambda<\frac{1}{\gamma}$, $\frac{1}{p}=\frac{1}{p_{1}%
}+\frac{1}{p_{2}}$, $\frac{1}{q}=\frac{1}{p}-\frac{\alpha}{\gamma}$, $\frac
{1}{q_{1}}=\frac{1}{p_{1}}-\frac{\alpha}{\gamma}$. Let also, for $s^{\prime
}\leq p$ the pair $(\varphi_{1},\varphi_{2})$ satisfies the condition%
\[
\int \limits_{r}^{\infty}\left(  1+\ln \frac{t}{r}\right)  \frac
{\operatorname*{essinf}\limits_{t<\tau<\infty}\varphi_{1}(x_{0},\tau
)\tau^{\frac{\gamma}{p_{1}}}}{t^{\frac{\gamma}{q_{1}}+1-\gamma \lambda}}dt\leq
C\, \varphi_{2}(x_{0},r),
\]
and for $q_{1}<s$ the pair $(\varphi_{1},\varphi_{2})$ satisfies the condition%
\[
\int \limits_{r}^{\infty}\left(  1+\ln \frac{t}{r}\right)  \frac
{\operatorname*{essinf}\limits_{t<\tau<\infty}\varphi_{1}(x_{0},\tau
)\tau^{\frac{\gamma}{p_{1}}}}{t^{\frac{\gamma}{q_{1}}-\frac{\gamma}%
{s}+1-\gamma \lambda}}dt\leq C\, \varphi_{2}(x_{0},r)r^{\frac{\gamma}{s}},
\]
where $C$ does not depend on $r$.

Then, the operators $[b,T_{\Omega,\alpha}^{P}]$ and $M_{\Omega,b,\alpha}^{P}$
are bounded from $LM_{p_{1},\varphi_{1},P}^{\{x_{0}\}}$ to $LM_{q,\varphi
_{2},P}^{\{x_{0}\}}$. Moreover,%
\[
\left \Vert \lbrack b,T_{\Omega,\alpha}^{P}]f\right \Vert _{LM_{q,\varphi_{2}%
,P}^{\{x_{0}\}}}\lesssim \left \Vert b\right \Vert _{LC_{p_{2},\lambda
,P}^{\left \{  x_{0}\right \}  }}\left \Vert f\right \Vert _{LM_{p_{1},\varphi
_{1},P}^{\{x_{0}\}}},
\]%
\[
\left \Vert M_{\Omega,b,\alpha}^{P}f\right \Vert _{LM_{q,\varphi_{2},P}%
^{\{x_{0}\}}}\lesssim \left \Vert b\right \Vert _{LC_{p_{2},\lambda,P}^{\left \{
x_{0}\right \}  }}\left \Vert f\right \Vert _{LM_{p_{1},\varphi_{1},P}%
^{\{x_{0}\}}}.
\]

\end{corollary}

At last, throughout the paper we use the letter $C$ for a positive constant,
independent of appropriate parameters and not necessarily the same at each
occurrence. By $A\lesssim B$ we mean that $A\leq CB$ with some positive
constant $C$ independent of appropriate quantities. If $A\lesssim B$ and
$B\lesssim A$, we write $A\approx B$ and say that $A$ and $B$ are equivalent.

\section{Some Lemmas}

To prove the main results (Theorems \ref{teo4} and \ref{teo4*}), we need the
following lemmas.

\begin{lemma}
\label{Lemma 4}\cite{Gurbuz2, Gurbuz3} Let $b$ be a parabolic local Campanato
function in $LC_{p,\lambda,P}^{\left \{  x_{0}\right \}  }\left(
\mathbb{R}
^{n}\right)  $, $1\leq p<\infty$, $0\leq \lambda<\frac{1}{\gamma}$ and $r_{1}$,
$r_{2}>0$. Then%
\begin{equation}
\left(  \frac{1}{\left \vert E\left(  x_{0},r_{1}\right)  \right \vert
^{1+\lambda p}}%
{\displaystyle \int \limits_{E\left(  x_{0},r_{1}\right)  }}
\left \vert b\left(  y\right)  -b_{E\left(  x_{0},r_{2}\right)  }\right \vert
^{p}dy\right)  ^{\frac{1}{p}}\leq C\left(  1+\ln \frac{r_{1}}{r_{2}}\right)
\left \Vert b\right \Vert _{LC_{p,\lambda,P}^{\left \{  x_{0}\right \}  }%
},\label{a}%
\end{equation}
where $C>0$ is independent of $b$, $r_{1}$ and $r_{2}$.

From this inequality $\left(  \text{\ref{a}}\right)  $, we have%
\begin{equation}
\left \vert b_{E\left(  x_{0},r_{1}\right)  }-b_{E\left(  x_{0},r_{2}\right)
}\right \vert \leq C\left(  1+\ln \frac{r_{1}}{r_{2}}\right)  \left \vert
E\left(  x_{0},r_{1}\right)  \right \vert ^{\lambda}\left \Vert b\right \Vert
_{LC_{p,\lambda,P}^{\left \{  x_{0}\right \}  }},\label{b}%
\end{equation}
and it is easy to see that%
\begin{equation}
\left \Vert b-b_{E}\right \Vert _{L_{p}\left(  E\right)  }\leq C\left(
1+\ln \frac{r_{1}}{r_{2}}\right)  r^{\frac{\gamma}{p}+\gamma \lambda}\left \Vert
b\right \Vert _{LC_{p,\lambda,P}^{\left \{  x_{0}\right \}  }}.\label{c}%
\end{equation}

\end{lemma}

\begin{lemma}
\label{lemma2}Suppose that $x_{0}\in{\mathbb{R}^{n}}$, $\Omega \in
L_{s}(S^{n-1})$, $1<s\leq \infty$, is $A_{t}$-homogeneous of degree zero. Let
$T_{\Omega}^{P}$ be a parabolic linear operator satisfying condition
(\ref{e1}). Let also $1<q,p_{i},p<\infty$ with $\frac{1}{q}=\sum
\limits_{i=1}^{m}\frac{1}{p_{i}}+\frac{1}{p}$ and $\overrightarrow{b}\in
LC_{p_{i},\lambda_{i},P}^{\left \{  x_{0}\right \}  }({\mathbb{R}^{n}})$ for
$0\leq \lambda_{i}<\frac{1}{\gamma}$, $i=1,\ldots,m$. Then, for $s^{\prime}\leq
q$ the inequality
\begin{equation}
\Vert \lbrack \overrightarrow{b},T_{\Omega}^{P}]f\Vert_{L_{q}(E(x_{0}%
,r))}\lesssim%
{\displaystyle \prod \limits_{i=1}^{m}}
\Vert \overrightarrow{b}\Vert_{LC_{p_{i},\lambda_{i},P}^{\left \{
x_{0}\right \}  }}r^{\frac{\gamma}{q}}%
{\displaystyle \int \limits_{2kr}^{\infty}}
\left(  1+\ln \frac{t}{r}\right)  ^{m}\frac{\Vert f\Vert_{L_{p}(E(x_{0},t))}%
}{t^{\gamma \left(  \frac{1}{p}-%
{\displaystyle \sum \limits_{i=1}^{m}}
\lambda_{i}\right)  +1}}dt\label{200}%
\end{equation}
holds for any ellipsoid $E\left(  x_{0},r\right)  $ and for all $f\in
L_{p}^{loc}({\mathbb{R}^{n}})$. Also, for $p<s$ the inequality%
\[
\Vert \lbrack \overrightarrow{b},T_{\Omega}^{P}]f\Vert_{L_{q}(E(x_{0}%
,r))}\lesssim%
{\displaystyle \prod \limits_{i=1}^{m}}
\Vert \overrightarrow{b}\Vert_{LC_{p_{i},\lambda_{i},P}^{\left \{
x_{0}\right \}  }}\,r^{\frac{\gamma}{q}-\frac{\gamma}{s}}%
{\displaystyle \int \limits_{2kr}^{\infty}}
\left(  1+\ln \frac{t}{r}\right)  ^{m}\frac{\Vert f\Vert_{L_{p}(E(x_{0},t))}%
}{t^{\gamma \left(  \frac{1}{p}-\frac{1}{s}-%
{\displaystyle \sum \limits_{i=1}^{m}}
\lambda_{i}\right)  +1}}dt
\]
holds for any ellipsoid $E\left(  x_{0},r\right)  $ and for all $f\in
L_{p}^{loc}({\mathbb{R}^{n}})$.
\end{lemma}

\begin{proof}
Without loss of generality, it is sufficient to show that the conclusion holds
for $[\overrightarrow{b},T_{\Omega}^{P}]f=[\left(  b_{1},b_{2}\right)
,T_{\Omega}^{P}]f$. Let $1<q,p_{i},p<\infty$ with $\frac{1}{q}=\sum
\limits_{i=1}^{m}\frac{1}{p_{i}}+\frac{1}{p}$ and $\overrightarrow{b}\in
LC_{p_{i},\lambda_{i},P}^{\left \{  x_{0}\right \}  }({\mathbb{R}^{n}})$ for
$0\leq \lambda_{i}<\frac{1}{\gamma}$, $i=1,\ldots,m$. Set $E=E\left(
x_{0},r\right)  $ for the parabolic ball (ellipsoid) centered at $x_{0}$ and
of radius $r$ and $2kE=E\left(  x_{0},2kr\right)  $. We represent $f$ as%
\begin{equation}
f=f_{1}+f_{2},\qquad \text{\ }f_{1}\left(  y\right)  =f\left(  y\right)
\chi_{2kE}\left(  y\right)  ,\qquad \text{\ }f_{2}\left(  y\right)  =f\left(
y\right)  \chi_{\left(  2kE\right)  ^{C}}\left(  y\right)  ,\qquad
r>0\label{9}%
\end{equation}
and thus have%
\[
\left \Vert \lbrack \left(  b_{1},b_{2}\right)  ,T_{\Omega}^{P}]f\right \Vert
_{L_{q}\left(  E\right)  }\leq \left \Vert \lbrack \left(  b_{1},b_{2}\right)
,T_{\Omega}^{P}]f_{1}\right \Vert _{L_{q}\left(  E\right)  }+\left \Vert
[\left(  b_{1},b_{2}\right)  ,T_{\Omega}^{P}]f_{2}\right \Vert _{L_{q}\left(
E\right)  }=:F+G.
\]
Let us estimate $F+G$, respectively.

For $[\left(  b_{1},b_{2}\right)  ,T_{\Omega}^{P}]f_{1}\left(  x\right)  $, we
have the following decomposition,%
\begin{align*}
\lbrack \left(  b_{1},b_{2}\right)  ,T_{\Omega}^{P}]f_{1}\left(  x\right)   &
=\left(  b_{1}\left(  x\right)  -\left(  b_{1}\right)  _{E}\right)  \left(
b_{2}\left(  x\right)  -\left(  b_{2}\right)  _{E}\right)  T_{\Omega}^{P}%
f_{1}\left(  x\right) \\
& -\left(  b_{1}\left(  \cdot \right)  -\left(  b_{1}\right)  _{E}\right)
T_{\Omega}^{P}\left(  \left(  b_{2}\left(  \cdot \right)  -\left(
b_{2}\right)  _{E}\right)  f_{1}\right)  \left(  x\right) \\
& +\left(  b_{2}\left(  x\right)  -\left(  b_{2}\right)  _{E}\right)
T_{\Omega}^{P}\left(  \left(  b_{1}\left(  x\right)  -\left(  b_{1}\right)
_{E}\right)  f_{1}\right)  \left(  x\right)  -\\
& T_{\Omega}^{P}\left(  \left(  b_{1}\left(  \cdot \right)  -\left(
b_{1}\right)  _{E}\right)  \left(  b_{2}\left(  \cdot \right)  -\left(
b_{2}\right)  _{E}\right)  f_{1}\right)  \left(  x\right)  .
\end{align*}
Hence, we get%
\begin{align}
F  & =\left \Vert [\left(  b_{1},b_{2}\right)  ,T_{\Omega}^{P}]f_{1}\right \Vert
_{L_{q}\left(  E\right)  }\lesssim \nonumber \\
& \left \Vert \left(  b_{1}-\left(  b_{1}\right)  _{E}\right)  \left(
b_{2}\left(  x\right)  -\left(  b_{2}\right)  _{E}\right)  T_{\Omega}^{P}%
f_{1}\right \Vert _{L_{q}\left(  E\right)  }\nonumber \\
& +\left \Vert \left(  b_{1}-\left(  b_{1}\right)  _{E}\right)  T_{\Omega}%
^{P}\left(  \left(  b_{2}-\left(  b_{2}\right)  _{E}\right)  f_{1}\right)
\right \Vert _{L_{q}\left(  E\right)  }\nonumber \\
& +\left \Vert \left(  b_{2}-\left(  b_{2}\right)  _{E}\right)  T_{\Omega}%
^{P}\left(  \left(  b_{1}-\left(  b_{1}\right)  _{E}\right)  f_{1}\right)
\right \Vert _{L_{q}\left(  E\right)  }\nonumber \\
& +\left \Vert T_{\Omega}^{P}\left(  \left(  b_{1}-\left(  b_{1}\right)
_{E}\right)  \left(  b_{2}-\left(  b_{2}\right)  _{E}\right)  f_{1}\right)
\right \Vert _{L_{q}\left(  E\right)  }\nonumber \\
& \equiv F_{1}+F_{2}+F_{3}+F_{4}.\label{0}%
\end{align}
One observes that the estimate of $F_{2}$ is analogous to that of $F_{3}$.
Thus, we will only estimate $F_{1}$, $F_{2}$ and $F_{4}$.

To estimate $F_{1}$, let $1<q,\tau<\infty$, such that $\frac{1}{q}=\frac
{1}{\tau}+\frac{1}{p}$, $\frac{1}{\tau}=\frac{1}{p_{1}}+\frac{1}{p_{2}}$.
Then, using H\"{o}lder's inequality and by the boundedness of $T_{\Omega}^{P}$
on $L_{p}$ (see Theorem 1.2. in \cite{Gurbuz2}) it follows that:%
\begin{align*}
F_{1}  & =\left \Vert \left(  b_{1}-\left(  b_{1}\right)  _{E}\right)  \left(
b_{2}\left(  x\right)  -\left(  b_{2}\right)  _{E}\right)  T_{\Omega}^{P}%
f_{1}\right \Vert _{L_{q}\left(  E\right)  }\\
& \lesssim \left \Vert \left(  b_{1}-\left(  b_{1}\right)  _{E}\right)  \left(
b_{2}\left(  x\right)  -\left(  b_{2}\right)  _{E}\right)  \right \Vert
_{L_{\tau}\left(  E\right)  }\left \Vert T_{\Omega}^{P}f_{1}\right \Vert
_{L_{p}\left(  E\right)  }\\
& \lesssim \left \Vert b_{1}-\left(  b_{1}\right)  _{E}\right \Vert _{L_{p_{1}%
}\left(  E\right)  }\left \Vert b_{2}-\left(  b_{2}\right)  _{E}\right \Vert
_{L_{p_{2}}\left(  E\right)  }\left \Vert f\right \Vert _{L_{p}\left(
2kE\right)  }\\
& \lesssim \left \Vert b_{1}-\left(  b_{1}\right)  _{E}\right \Vert _{L_{p_{1}%
}\left(  E\right)  }\left \Vert b_{2}-\left(  b_{2}\right)  _{E}\right \Vert
_{L_{p_{2}}\left(  E\right)  }r^{\frac{\gamma}{p}}\int \limits_{2kr}^{\infty
}\Vert f\Vert_{L_{p}(E(x_{0},t))}\frac{dt}{t^{\frac{\gamma}{p}+1}}.
\end{align*}
From Lemma \ref{Lemma 4}, it is easy to see that%
\begin{equation}
\left \Vert b_{i}-\left(  b_{i}\right)  _{E}\right \Vert _{L_{p_{i}}\left(
E\right)  }\leq Cr^{\frac{\gamma}{p_{i}}+\gamma \lambda_{i}}\left \Vert
b_{i}\right \Vert _{LC_{p_{i},\lambda_{i},P}^{\left \{  x_{0}\right \}  }%
},\label{1}%
\end{equation}
and%
\begin{align}
\left \Vert b_{i}-\left(  b_{i}\right)  _{E}\right \Vert _{L_{p_{i}}\left(
2kE\right)  } &  \leq \left \Vert b_{i}-\left(  b_{i}\right)  _{2kE}\right \Vert
_{L_{p_{i}}\left(  2kE\right)  }+\left \Vert \left(  b_{i}\right)  _{E}-\left(
b_{i}\right)  _{2kE}\right \Vert _{L_{p_{i}}\left(  2kE\right)  }\nonumber \\
&  \lesssim r^{\frac{\gamma}{p_{i}}+\gamma \lambda_{i}}\left \Vert
b_{i}\right \Vert _{LC_{p_{i},\lambda_{i},P}^{\left \{  x_{0}\right \}  }%
},\label{2}%
\end{align}
for $i=1$, $2$. Hence, by (\ref{1}) we get%
\begin{align*}
F_{1}  & \lesssim \Vert b_{1}\Vert_{LC_{p_{1},\lambda_{1},P}^{\left \{
x_{0}\right \}  }}\Vert b_{2}\Vert_{LC_{p_{2},\lambda_{2},P}^{\left \{
x_{0}\right \}  }}r^{\gamma \left(  \frac{1}{p_{1}}+\frac{1}{p_{2}}+\frac{1}%
{p}\right)  }\int \limits_{2kr}^{\infty}\left(  1+\ln \frac{t}{r}\right)
^{2}t^{-\frac{\gamma}{p}+\gamma \left(  \lambda_{1}+\lambda_{2}\right)
-1}\Vert f\Vert_{L_{p}(E(x_{0},t))}dt\\
& \lesssim \Vert b_{1}\Vert_{LC_{p_{1},\lambda_{1},P}^{\left \{  x_{0}\right \}
}}\Vert b_{2}\Vert_{LC_{p_{2},\lambda_{2},P}^{\left \{  x_{0}\right \}  }%
}r^{\frac{\gamma}{q}}\int \limits_{2kr}^{\infty}\left(  1+\ln \frac{t}%
{r}\right)  ^{2}\frac{\Vert f\Vert_{L_{p}(E(x_{0},t))}}{t^{\frac{\gamma}%
{p}-\gamma \left(  \lambda_{1}+\lambda_{2}\right)  +1}}dt.
\end{align*}
To estimate $F_{2}$, let $1<\tau<\infty$, such that $\frac{1}{q}=\frac
{1}{p_{1}}+\frac{1}{\tau}$. Then, similar to the estimates for $F_{1}$, we
have%
\begin{align*}
F_{2} &  \lesssim \left \Vert b_{1}-\left(  b_{1}\right)  _{E}\right \Vert
_{L_{p_{1}}\left(  E\right)  }\left \Vert T_{\Omega}^{P}\left(  \left(
b_{2}\left(  \cdot \right)  -\left(  b_{2}\right)  _{E}\right)  f_{1}\right)
\right \Vert _{L_{\tau}\left(  E\right)  }\\
&  \lesssim \left \Vert b_{1}-\left(  b_{1}\right)  _{E}\right \Vert _{L_{p_{1}%
}\left(  E\right)  }\left \Vert \left(  b_{2}\left(  \cdot \right)  -\left(
b_{2}\right)  _{E}\right)  f_{1}\right \Vert _{L_{k}\left(  E\right)  }\\
&  \lesssim \left \Vert b_{1}-\left(  b_{1}\right)  _{E}\right \Vert _{L_{p_{1}%
}\left(  E\right)  }\left \Vert b_{2}-\left(  b_{2}\right)  _{E}\right \Vert
_{L_{p_{2}}\left(  2kE\right)  }\left \Vert f\right \Vert _{L_{p}\left(
2kE\right)  },
\end{align*}
where $1<k<\infty$, such that $\frac{1}{k}=\frac{1}{p_{2}}+\frac{1}{p}%
=\frac{1}{\tau}$. By (\ref{1}) and (\ref{2}), we get%
\[
F_{2}\lesssim \Vert b_{1}\Vert_{LC_{p_{1},\lambda_{1},P}^{\left \{
x_{0}\right \}  }}\Vert b_{2}\Vert_{LC_{p_{2},\lambda_{2},P}^{\left \{
x_{0}\right \}  }}r^{\frac{\gamma}{q}}\int \limits_{2kr}^{\infty}\left(
1+\ln \frac{t}{r}\right)  ^{2}\frac{\Vert f\Vert_{L_{p}(E(x_{0},t))}}%
{t^{\frac{\gamma}{p}-\gamma \left(  \lambda_{1}+\lambda_{2}\right)  +1}}dt.
\]
In a similar way, $F_{3}$ has the same estimate as above, so we omit the
details. Then we have that%
\[
F_{3}\lesssim \Vert b_{1}\Vert_{LC_{p_{1},\lambda_{1},P}^{\left \{
x_{0}\right \}  }}\Vert b_{2}\Vert_{LC_{p_{2},\lambda_{2},P}^{\left \{
x_{0}\right \}  }}r^{\frac{\gamma}{q}}\int \limits_{2kr}^{\infty}\left(
1+\ln \frac{t}{r}\right)  ^{2}\frac{\Vert f\Vert_{L_{p}(E(x_{0},t))}}%
{t^{\frac{\gamma}{p}-\gamma \left(  \lambda_{1}+\lambda_{2}\right)  +1}}dt.
\]
Now let us consider the term $F_{4}$. Let $1<q,\tau<\infty$, such that
$\frac{1}{q}=\frac{1}{\tau}+\frac{1}{p}$, $\frac{1}{\tau}=\frac{1}{p_{1}%
}+\frac{1}{p_{2}}$. Then by the boundedness of $T_{\Omega}^{P}$ on $L_{p}$
(see Theorem 1.2. in \cite{Gurbuz2}), H\"{o}lder's inequality and (\ref{2}),
we obtain%
\begin{align*}
F_{4}  & =\left \Vert T_{\Omega}^{P}\left(  \left(  b_{1}-\left(  b_{1}\right)
_{E}\right)  \left(  b_{2}-\left(  b_{2}\right)  _{E}\right)  f_{1}\right)
\right \Vert _{L_{q}\left(  E\right)  }\\
& \lesssim \left \Vert \left(  b_{1}-\left(  b_{1}\right)  _{E}\right)  \left(
b_{2}-\left(  b_{2}\right)  _{E}\right)  f_{1}\right \Vert _{L_{q}\left(
E\right)  }\\
& \lesssim \left \Vert \left(  b_{1}-\left(  b_{1}\right)  _{E}\right)  \left(
b_{2}-\left(  b_{2}\right)  _{E}\right)  \right \Vert _{L_{\tau}\left(
E\right)  }\left \Vert f_{1}\right \Vert _{L_{p}\left(  E\right)  }\\
& \lesssim \left \Vert b_{1}-\left(  b_{1}\right)  _{E}\right \Vert _{L_{p_{1}%
}\left(  2kE\right)  }\left \Vert b_{2}-\left(  b_{2}\right)  _{E}\right \Vert
_{L_{p_{2}}\left(  2kE\right)  }\left \Vert f\right \Vert _{L_{p}\left(
2kE\right)  }\\
& \lesssim \Vert b_{1}\Vert_{LC_{p_{1},\lambda_{1},P}^{\left \{  x_{0}\right \}
}}\Vert b_{2}\Vert_{LC_{p_{2},\lambda_{2},P}^{\left \{  x_{0}\right \}  }%
}r^{\frac{\gamma}{q}}\int \limits_{2kr}^{\infty}\left(  1+\ln \frac{t}%
{r}\right)  ^{2}\frac{\Vert f\Vert_{L_{p}(E(x_{0},t))}}{t^{\frac{\gamma}%
{p}-\gamma \left(  \lambda_{1}+\lambda_{2}\right)  +1}}dt.
\end{align*}
Combining all the estimates of $F_{1}$, $F_{2}$, $F_{3}$, $F_{4}$; we get%
\[
F=\left \Vert [\left(  b_{1},b_{2}\right)  ,T_{\Omega}^{P}]f_{1}\right \Vert
_{L_{q}\left(  E\right)  }\lesssim \Vert b_{1}\Vert_{LC_{p_{1},\lambda_{1}%
,P}^{\left \{  x_{0}\right \}  }}\Vert b_{2}\Vert_{LC_{p_{2},\lambda_{2}%
,P}^{\left \{  x_{0}\right \}  }}r^{\frac{\gamma}{q}}\int \limits_{2kr}^{\infty
}\left(  1+\ln \frac{t}{r}\right)  ^{2}\frac{\Vert f\Vert_{L_{p}(E(x_{0},t))}%
}{t^{\frac{\gamma}{p}-\gamma \left(  \lambda_{1}+\lambda_{2}\right)  +1}}dt.
\]

Now, let us estimate $G=\left \Vert [\left(  b_{1},b_{2}\right)  ,T_{\Omega
}^{P}]f_{2}\right \Vert _{L_{q}\left(  E\right)  }$. For $G$, it's similar to
(\ref{0}) we also write%
\begin{align*}
G  & =\left \Vert [\left(  b_{1},b_{2}\right)  ,T_{\Omega}^{P}]f_{2}\right \Vert
_{L_{q}\left(  E\right)  }\lesssim \\
& \left \Vert \left(  b_{1}-\left(  b_{1}\right)  _{E}\right)  \left(
b_{2}\left(  x\right)  -\left(  b_{2}\right)  _{E}\right)  T_{\Omega}^{P}%
f_{2}\right \Vert _{L_{q}\left(  E\right)  }\\
& +\left \Vert \left(  b_{1}-\left(  b_{1}\right)  _{E}\right)  T_{\Omega}%
^{P}\left(  \left(  b_{2}-\left(  b_{2}\right)  _{E}\right)  f_{2}\right)
\right \Vert _{L_{q}\left(  E\right)  }\\
& +\left \Vert \left(  b_{2}-\left(  b_{2}\right)  _{E}\right)  T_{\Omega}%
^{P}\left(  \left(  b_{1}-\left(  b_{1}\right)  _{E}\right)  f_{2}\right)
\right \Vert _{L_{q}\left(  E\right)  }\\
& +\left \Vert T_{\Omega}^{P}\left(  \left(  b_{1}-\left(  b_{1}\right)
_{E}\right)  \left(  b_{2}-\left(  b_{2}\right)  _{E}\right)  f_{2}\right)
\right \Vert _{L_{q}\left(  E\right)  }\\
& \equiv G_{1}+G_{2}+G_{3}+G_{4}.
\end{align*}
To estimate $G_{1}$, let $1<q,\tau<\infty$, such that $\frac{1}{q}=\frac
{1}{\tau}+\frac{1}{p}$, $\frac{1}{\tau}=\frac{1}{p_{1}}+\frac{1}{p_{2}}$.
Then, using H\"{o}lder's inequality and by (\ref{1}), we have%
\begin{align*}
G_{1}  & =\left \Vert \left(  b_{1}-\left(  b_{1}\right)  _{E}\right)  \left(
b_{2}\left(  x\right)  -\left(  b_{2}\right)  _{E}\right)  T_{\Omega}^{P}%
f_{2}\right \Vert _{L_{q}\left(  E\right)  }\\
& \lesssim \left \Vert \left(  b_{1}-\left(  b_{1}\right)  _{E}\right)  \left(
b_{2}-\left(  b_{2}\right)  _{E}\right)  \right \Vert _{L_{\tau}\left(
E\right)  }\left \Vert T_{\Omega}^{P}f_{2}\right \Vert _{L_{p}\left(  E\right)
}\\
& \lesssim \left \Vert b_{1}-\left(  b_{1}\right)  _{E}\right \Vert _{L_{p_{1}%
}\left(  E\right)  }\left \Vert b_{2}-\left(  b_{2}\right)  _{E}\right \Vert
_{L_{p_{2}}\left(  E\right)  }r^{\frac{\gamma}{_{p}}}\int \limits_{2kr}%
^{\infty}\Vert f\Vert_{L_{p}(E(x_{0},t))}t^{-\frac{\gamma}{p}-1}dt\\
& \lesssim \Vert b_{1}\Vert_{LC_{p_{1},\lambda_{1},P}^{\left \{  x_{0}\right \}
}}\Vert b_{2}\Vert_{LC_{p_{2},\lambda_{2},P}^{\left \{  x_{0}\right \}  }%
}r^{\gamma \left(  \frac{1}{p_{1}}+\frac{1}{p_{2}}+\frac{1}{p}\right)  }\\
& \times \int \limits_{2kr}^{\infty}\left(  1+\ln \frac{t}{r}\right)
^{2}t^{-\frac{\gamma}{p}+\gamma \left(  \lambda_{1}+\lambda_{2}\right)
-1}\Vert f\Vert_{L_{p}(E(x_{0},t))}dt\\
& \lesssim \Vert b_{1}\Vert_{LC_{p_{1},\lambda_{1},P}^{\left \{  x_{0}\right \}
}}\Vert b_{2}\Vert_{LC_{p_{2},\lambda_{2},P}^{\left \{  x_{0}\right \}  }%
}r^{\frac{\gamma}{q}}\int \limits_{2kr}^{\infty}\left(  1+\ln \frac{t}%
{r}\right)  ^{2}\frac{\Vert f\Vert_{L_{p}(E(x_{0},t))}}{t^{\frac{\gamma}%
{p}-\gamma \left(  \lambda_{1}+\lambda_{2}\right)  +1}}dt,
\end{align*}
where in the second inequality we have used the following fact:

It is clear that $x\in E$, $y\in \left(  2kE\right)  ^{C}$ implies
\begin{equation}
\frac{1}{2k}\rho \left(  x_{0}-y\right)  \leq \rho \left(  x-y\right)  \leq
\frac{3k}{2}\rho \left(  x_{0}-y\right)  .\label{7}%
\end{equation}
Hence, we get%
\[
\left \vert T_{\Omega}^{P}f_{2}\left(  x\right)  \right \vert \leq2^{\gamma
}c_{1}\int \limits_{\left(  2kE\right)  ^{C}}\frac{\left \vert f\left(
y\right)  \right \vert \left \vert \Omega \left(  x-y\right)  \right \vert }%
{\rho \left(  x_{0}-y\right)  ^{\gamma}}dy.
\]

By Fubini's theorem, we have%
\begin{align*}
\int \limits_{\left(  2kE\right)  ^{C}}\frac{\left \vert f\left(  y\right)
\right \vert \left \vert \Omega \left(  x-y\right)  \right \vert }{\rho \left(
x_{0}-y\right)  ^{\gamma}}dy  & \approx \int \limits_{\left(  2kE\right)  ^{C}%
}\left \vert f\left(  y\right)  \right \vert \left \vert \Omega \left(
x-y\right)  \right \vert \int \limits_{\rho \left(  x_{0}-y\right)  }^{\infty
}\frac{dt}{t^{\gamma+1}}dy\\
& \approx \int \limits_{2kr}^{\infty}\int \limits_{2kr\leq \rho \left(
x_{0}-y\right)  \leq t}\left \vert f\left(  y\right)  \right \vert \left \vert
\Omega \left(  x-y\right)  \right \vert dy\frac{dt}{t^{\gamma+1}}\\
& \lesssim \int \limits_{2kr}^{\infty}\int \limits_{E\left(  x_{0},t\right)
}\left \vert f\left(  y\right)  \right \vert \left \vert \Omega \left(
x-y\right)  \right \vert dy\frac{dt}{t^{\gamma+1}}.
\end{align*}
Applying H\"{o}lder's inequality, we get%
\begin{align}
& \int \limits_{\left(  2kE\right)  ^{C}}\frac{\left \vert f\left(  y\right)
\right \vert \left \vert \Omega \left(  x-y\right)  \right \vert }{\rho \left(
x_{0}-y\right)  ^{\gamma}}dy\nonumber \\
& \lesssim \int \limits_{2kr}^{\infty}\left \Vert f\right \Vert _{L_{p}\left(
E\left(  x_{0},t\right)  \right)  }\left \Vert \Omega \left(  x-\cdot \right)
\right \Vert _{L_{s}\left(  E\left(  x_{0},t\right)  \right)  }\left \vert
E\left(  x_{0},t\right)  \right \vert ^{1-\frac{1}{p}-\frac{1}{s}}\frac
{dt}{t^{\gamma+1}}.\label{e310}%
\end{align}

For $x\in E\left(  x_{0},t\right)  $, notice that $\Omega$ is $A_{t}%
$-homogenous of degree zero and $\Omega \in L_{s}(S^{n-1})$, $s>1$. Then, we
obtain%
\begin{align}
\left(  \int \limits_{E\left(  x_{0},t\right)  }\left \vert \Omega \left(
x-y\right)  \right \vert ^{s}dy\right)  ^{\frac{1}{s}}  & =\left(
\int \limits_{E\left(  x-x_{0},t\right)  }\left \vert \Omega \left(  z\right)
\right \vert ^{s}dz\right)  ^{\frac{1}{s}}\nonumber \\
& \leq \left(  \int \limits_{E\left(  0,t+\left \vert x-x_{0}\right \vert \right)
}\left \vert \Omega \left(  z\right)  \right \vert ^{s}dz\right)  ^{\frac{1}{s}%
}\nonumber \\
& \leq \left(  \int \limits_{E\left(  0,2t\right)  }\left \vert \Omega \left(
z\right)  \right \vert ^{s}dz\right)  ^{\frac{1}{s}}\nonumber \\
& =\left(  \int \limits_{S^{n-1}}\int \limits_{0}^{2t}\left \vert \Omega \left(
z^{\prime}\right)  \right \vert ^{s}d\sigma \left(  z^{\prime}\right)
r^{n-1}dr\right)  ^{\frac{1}{s}}\nonumber \\
& =C\left \Vert \Omega \right \Vert _{L_{s}\left(  S^{n-1}\right)  }\left \vert
E\left(  x_{0},2t\right)  \right \vert ^{\frac{1}{s}}.\label{e311}%
\end{align}
Thus, by (\ref{e311}), it follows that:%
\[
\left \vert T_{\Omega}^{P}f_{2}\left(  x\right)  \right \vert \lesssim
\int \limits_{2kr}^{\infty}\left \Vert f\right \Vert _{L_{p}\left(  E\left(
x_{0},t\right)  \right)  }\frac{dt}{t^{\frac{\gamma}{p}+1}}.
\]

Moreover, for all $p\in \left[  1,\infty \right)  $ the inequality%
\[
\left \Vert T_{\Omega}^{P}f_{2}\right \Vert _{L_{p}\left(  E\right)  }\lesssim
r^{\frac{\gamma}{p}}\int \limits_{2kr}^{\infty}\left \Vert f\right \Vert
_{L_{p}\left(  E\left(  x_{0},t\right)  \right)  }\frac{dt}{t^{\frac{\gamma
}{p}+1}}%
\]
is valid.

On the other hand, for the estimates used in $G_{2}$, $G_{3}$, we have to
prove the below inequality:%
\begin{equation}
\left \vert T_{\Omega}^{P}\left(  \left(  b_{2}-\left(  b_{2}\right)
_{E}\right)  f_{2}\right)  \left(  x\right)  \right \vert \lesssim \left \Vert
b_{2}\right \Vert _{LC_{p_{2},\lambda_{2},P}^{\left \{  x_{0}\right \}  }}%
{\displaystyle \int \limits_{2kr}^{\infty}}
\left(  1+\ln \frac{t}{r}\right)  t^{-\frac{\gamma}{p}+\gamma \lambda_{2}%
-1}\left \Vert f\right \Vert _{L_{p}\left(  E\left(  x_{0},t\right)  \right)
}dt.\label{11}%
\end{equation}
Indeed, when $s^{\prime}\leq q$, for $x\in E$, by Fubini's theorem and
applying H\"{o}lder's inequality and from (\ref{b}), (\ref{c}), (\ref{e311})
we have

$\left \vert T_{\Omega}^{P}\left(  \left(  b_{2}\left(  \cdot \right)  -\left(
b_{2}\right)  _{E}\right)  f_{2}\right)  \left(  x\right)  \right \vert
\lesssim%
{\displaystyle \int \limits_{\left(  2kE\right)  ^{C}}}
\left \vert b_{2}\left(  y\right)  -\left(  b_{2}\right)  _{E}\right \vert
\left \vert \Omega \left(  x-y\right)  \right \vert \frac{\left \vert f\left(
y\right)  \right \vert }{\rho \left(  x-y\right)  ^{\gamma}}dy$

$\lesssim%
{\displaystyle \int \limits_{\left(  2kE\right)  ^{C}}}
\left \vert b_{2}\left(  y\right)  -\left(  b_{2}\right)  _{E}\right \vert
\left \vert \Omega \left(  x-y\right)  \right \vert \frac{\left \vert f\left(
y\right)  \right \vert }{\rho \left(  x_{0}-y\right)  ^{\gamma}}dy$

$\approx%
{\displaystyle \int \limits_{2kr}^{\infty}}
{\displaystyle \int \limits_{2kr<\rho \left(  x_{0}-y\right)  <t}}
\left \vert b_{2}\left(  y\right)  -\left(  b_{2}\right)  _{E}\right \vert
\left \vert \Omega \left(  x-y\right)  \right \vert \left \vert f\left(  y\right)
\right \vert dy\frac{dt}{t^{\gamma+1}}$

$\lesssim%
{\displaystyle \int \limits_{2kr}^{\infty}}
{\displaystyle \int \limits_{E\left(  x_{0},t\right)  }}
\left \vert b_{2}\left(  y\right)  -\left(  b_{2}\right)  _{E\left(
x_{0},t\right)  }\right \vert \left \vert \Omega \left(  x-y\right)  \right \vert
\left \vert f\left(  y\right)  \right \vert dy\frac{dt}{t^{\gamma+1}}$

$+%
{\displaystyle \int \limits_{2kr}^{\infty}}
\left \vert \left(  b_{2}\right)  _{E\left(  x_{0},r\right)  }-\left(
b_{2}\right)  _{E\left(  x_{0},t\right)  }\right \vert
{\displaystyle \int \limits_{E\left(  x_{0},t\right)  }}
\left \vert \Omega \left(  x-y\right)  \right \vert \left \vert f\left(  y\right)
\right \vert dy\frac{dt}{t^{\gamma+1}}$

$\lesssim%
{\displaystyle \int \limits_{2kr}^{\infty}}
\left \Vert b_{2}\left(  \cdot \right)  -\left(  b_{2}\right)  _{E\left(
x_{0},t\right)  }\right \Vert _{L_{p_{2}}\left(  E\left(  x_{0},t\right)
\right)  }\left \Vert \Omega \left(  \cdot-y\right)  \right \Vert _{L_{s}\left(
E\left(  x_{0},t\right)  \right)  }\left \Vert f\right \Vert _{L_{p}\left(
E\left(  x_{0},t\right)  \right)  }\left \vert E\left(  x_{0},t\right)
\right \vert ^{1-\frac{1}{p_{2}}-\frac{1}{s}-\frac{1}{p}}\frac{dt}{t^{\gamma
+1}}$

$+%
{\displaystyle \int \limits_{2kr}^{\infty}}
\left \vert \left(  b_{2}\right)  _{E\left(  x_{0},r\right)  }-\left(
b_{2}\right)  _{E\left(  x_{0},t\right)  }\right \vert \left \Vert f\right \Vert
_{L_{p}\left(  E\left(  x_{0},t\right)  \right)  }\left \Vert \Omega \left(
\cdot-y\right)  \right \Vert _{L_{s}\left(  E\left(  x_{0},t\right)  \right)
}\left \vert E\left(  x_{0},t\right)  \right \vert ^{1-\frac{1}{p}-\frac{1}{s}%
}\frac{dt}{t^{\gamma+1}}$

$\lesssim%
{\displaystyle \int \limits_{2kr}^{\infty}}
\left \Vert b_{2}\left(  \cdot \right)  -\left(  b_{2}\right)  _{E\left(
x_{0},t\right)  }\right \Vert _{L_{p_{2}}\left(  E\left(  x_{0},t\right)
\right)  }\left \Vert f\right \Vert _{L_{p}\left(  E\left(  x_{0},t\right)
\right)  }t^{-1-\frac{\gamma}{p_{2}}-\frac{\gamma}{p}}dt$

$+\left \Vert b_{2}\right \Vert _{LC_{p_{2},\lambda_{2},P}^{\left \{
x_{0}\right \}  }}%
{\displaystyle \int \limits_{2kr}^{\infty}}
\left(  1+\ln \frac{t}{r}\right)  \left \Vert f\right \Vert _{L_{p}\left(
E\left(  x_{0},t\right)  \right)  }t^{-1-\frac{\gamma}{p}+\gamma \lambda_{2}%
}dt$

$\lesssim \left \Vert b_{2}\right \Vert _{LC_{p_{2},\lambda_{2},P}^{\left \{
x_{0}\right \}  }}%
{\displaystyle \int \limits_{2kr}^{\infty}}
\left(  1+\ln \frac{t}{r}\right)  t^{-\frac{\gamma}{p}+\gamma \lambda_{2}%
-1}\left \Vert f\right \Vert _{L_{p}\left(  E\left(  x_{0},t\right)  \right)
}dt.$

This completes the proof of inequality (\ref{11}).

Let $1<\tau<\infty$, such that $\frac{1}{q}=\frac{1}{p_{1}}+\frac{1}{\tau} $.
Then, using H\"{o}lder's inequality and from (\ref{11}) and (\ref{c}), we get%
\begin{align*}
G_{2} &  =\left \Vert \left(  b_{1}-\left(  b_{1}\right)  _{E}\right)
T_{\Omega}^{P}\left(  \left(  b_{2}-\left(  b_{2}\right)  _{E}\right)
f_{2}\right)  \right \Vert _{L_{q}\left(  E\right)  }\\
&  \lesssim \left \Vert b_{1}-\left(  b_{1}\right)  _{E}\right \Vert _{L_{p_{1}%
}\left(  E\right)  }\left \Vert T_{\Omega}^{P}\left(  \left(  b_{2}\left(
\cdot \right)  -\left(  b_{2}\right)  _{E}\right)  f_{2}\right)  \right \Vert
_{L_{\tau}\left(  E\right)  }\\
&  \lesssim \Vert b_{1}\Vert_{LC_{p_{1},\lambda_{1},P}^{\left \{  x_{0}\right \}
}}\Vert b_{2}\Vert_{LC_{p_{2},\lambda_{2},P}^{\left \{  x_{0}\right \}  }%
}r^{\frac{\gamma}{q}}\int \limits_{2kr}^{\infty}\left(  1+\ln \frac{t}%
{r}\right)  ^{2}\frac{\Vert f\Vert_{L_{p}(E(x_{0},t))}}{t^{\frac{\gamma}%
{p}-\gamma \left(  \lambda_{1}+\lambda_{2}\right)  +1}}dt.
\end{align*}
Similarly, $G_{3}$ has the same estimate above, so here we omit the details.
Then the inequality%
\begin{align*}
G_{3}  & =\left \Vert \left(  b_{2}-\left(  b_{2}\right)  _{E}\right)
T_{\Omega}^{P}\left(  \left(  b_{1}-\left(  b_{1}\right)  _{E}\right)
f_{2}\right)  \right \Vert _{L_{q}\left(  E\right)  }\\
& \lesssim \Vert b_{1}\Vert_{LC_{p_{1},\lambda_{1},P}^{\left \{  x_{0}\right \}
}}\Vert b_{2}\Vert_{LC_{p_{2},\lambda_{2},P}^{\left \{  x_{0}\right \}  }%
}r^{\frac{\gamma}{q}}\int \limits_{2kr}^{\infty}\left(  1+\ln \frac{t}%
{r}\right)  ^{2}\frac{\Vert f\Vert_{L_{p}(E(x_{0},t))}}{t^{\frac{\gamma}%
{p}-\gamma \left(  \lambda_{1}+\lambda_{2}\right)  +1}}dt
\end{align*}
is valid.

Now, let us estimate $G_{4}=\left \Vert T_{\Omega}^{P}\left(  \left(
b_{1}-\left(  b_{1}\right)  _{E}\right)  \left(  b_{2}-\left(  b_{2}\right)
_{E}\right)  f_{2}\right)  \right \Vert _{L_{q}\left(  E\right)  }$. It's
similar to the estimate of (\ref{11}), for any $x\in E$, we also write

$\left \vert T_{\Omega}^{P}\left(  \left(  b_{1}-\left(  b_{1}\right)
_{E}\right)  \left(  b_{2}-\left(  b_{2}\right)  _{E}\right)  f_{2}\right)
\left(  x\right)  \right \vert $

$\lesssim%
{\displaystyle \int \limits_{2kr}^{\infty}}
{\displaystyle \int \limits_{E\left(  x_{0},t\right)  }}
\left \vert b_{1}\left(  y\right)  -\left(  b_{1}\right)  _{E\left(
x_{0},t\right)  }\right \vert \left \vert b_{2}\left(  y\right)  -\left(
b_{2}\right)  _{E\left(  x_{0},t\right)  }\right \vert \left \vert
\Omega(x-y)\right \vert \left \vert f\left(  y\right)  \right \vert dy\frac
{dt}{t^{\gamma+1}}$

$+%
{\displaystyle \int \limits_{2kr}^{\infty}}
{\displaystyle \int \limits_{E\left(  x_{0},t\right)  }}
\left \vert b_{1}\left(  y\right)  -\left(  b_{1}\right)  _{E\left(
x_{0},t\right)  }\right \vert \left \vert \left(  b_{2}\right)  _{E\left(
x_{0},t\right)  }-\left(  b_{2}\right)  _{E\left(  x_{0},r\right)
}\right \vert \left \vert \Omega(x-y)\right \vert \left \vert f\left(  y\right)
\right \vert dy\frac{dt}{t^{\gamma+1}}$

$+%
{\displaystyle \int \limits_{2kr}^{\infty}}
{\displaystyle \int \limits_{E\left(  x_{0},t\right)  }}
\left \vert \left(  b_{1}\right)  _{E\left(  x_{0},t\right)  }-\left(
b_{2}\right)  _{E\left(  x_{0},r\right)  }\right \vert \left \vert b_{2}\left(
y\right)  -\left(  b_{2}\right)  _{E\left(  x_{0},t\right)  }\right \vert
\left \vert \Omega(x-y)\right \vert \left \vert f\left(  y\right)  \right \vert
dy\frac{dt}{t^{\gamma+1}}$

$+%
{\displaystyle \int \limits_{2kr}^{\infty}}
{\displaystyle \int \limits_{E\left(  x_{0},t\right)  }}
\left \vert \left(  b_{1}\right)  _{E\left(  x_{0},t\right)  }-\left(
b_{2}\right)  _{E\left(  x_{0},r\right)  }\right \vert \left \vert \left(
b_{2}\right)  _{E\left(  x_{0},t\right)  }-\left(  b_{2}\right)  _{E\left(
x_{0},r\right)  }\right \vert \left \vert \Omega(x-y)\right \vert \left \vert
f\left(  y\right)  \right \vert dy\frac{dt}{t^{\gamma+1}}$

$\equiv G_{41}+G_{42}+G_{43}+G_{44}.$

Let us estimate $G_{41}$, $G_{42}$, $G_{43}$, $G_{44}$, respectively.

Firstly, to estimate $G_{41}$, similar to the estimate of (\ref{11}), we get%
\[
G_{41}\lesssim \Vert b_{1}\Vert_{LC_{p_{1},\lambda_{1},P}^{\left \{
x_{0}\right \}  }}\Vert b_{2}\Vert_{LC_{p_{2},\lambda_{2},P}^{\left \{
x_{0}\right \}  }}\int \limits_{2kr}^{\infty}\left(  1+\ln \frac{t}{r}\right)
^{2}\frac{\Vert f\Vert_{L_{p}(E(x_{0},t))}}{t^{\frac{\gamma}{p}-\gamma \left(
\lambda_{1}+\lambda_{2}\right)  +1}}dt.
\]
Secondly, to estimate $G_{42}$ and $G_{43}$, from (\ref{11}), (\ref{b}) and
(\ref{c}), it follows that%
\[
G_{42}\lesssim \Vert b_{1}\Vert_{LC_{p_{1},\lambda_{1},P}^{\left \{
x_{0}\right \}  }}\Vert b_{2}\Vert_{LC_{p_{2},\lambda_{2},P}^{\left \{
x_{0}\right \}  }}\int \limits_{2kr}^{\infty}\left(  1+\ln \frac{t}{r}\right)
^{2}\frac{\Vert f\Vert_{L_{p}(E(x_{0},t))}}{t^{\frac{\gamma}{p}-\gamma \left(
\lambda_{1}+\lambda_{2}\right)  +1}}dt,
\]
and%
\[
G_{43}\lesssim \Vert b_{1}\Vert_{LC_{p_{1},\lambda_{1},P}^{\left \{
x_{0}\right \}  }}\Vert b_{2}\Vert_{LC_{p_{2},\lambda_{2},P}^{\left \{
x_{0}\right \}  }}\int \limits_{2kr}^{\infty}\left(  1+\ln \frac{t}{r}\right)
^{2}\frac{\Vert f\Vert_{L_{p}(E(x_{0},t))}}{t^{\frac{\gamma}{p}-\gamma \left(
\lambda_{1}+\lambda_{2}\right)  +1}}dt.
\]
Finally, to estimate $G_{44}$, similar to the estimate of (\ref{11}) and from
(\ref{b}) and (\ref{c}), we have%
\[
G_{44}\lesssim \Vert b_{1}\Vert_{LC_{p_{1},\lambda_{1},P}^{\left \{
x_{0}\right \}  }}\Vert b_{2}\Vert_{LC_{p_{2},\lambda_{2},P}^{\left \{
x_{0}\right \}  }}\int \limits_{2kr}^{\infty}\left(  1+\ln \frac{t}{r}\right)
^{2}\frac{\Vert f\Vert_{L_{p}(E(x_{0},t))}}{t^{\frac{\gamma}{p}-\gamma \left(
\lambda_{1}+\lambda_{2}\right)  +1}}dt.
\]
By the estimates of $G_{4j}$ above, where $j=1$, $2$, $3$, we know that
\[
\left \vert T_{\Omega}^{P}\left(  \left(  b_{1}-\left(  b_{1}\right)
_{E}\right)  \left(  b_{2}-\left(  b_{2}\right)  _{E}\right)  f_{2}\right)
\left(  x\right)  \right \vert \lesssim \Vert b_{1}\Vert_{LC_{p_{1},\lambda
_{1},P}^{\left \{  x_{0}\right \}  }}\Vert b_{2}\Vert_{LC_{p_{2},\lambda_{2}%
,P}^{\left \{  x_{0}\right \}  }}\int \limits_{2kr}^{\infty}\left(  1+\ln \frac
{t}{r}\right)  ^{2}\frac{\Vert f\Vert_{L_{p}(E(x_{0},t))}}{t^{\frac{\gamma}%
{p}-\gamma \left(  \lambda_{1}+\lambda_{2}\right)  +1}}dt.
\]
Then, we have%
\begin{align*}
G_{4}  & =\left \Vert T_{\Omega}^{P}\left(  \left(  b_{1}-\left(  b_{1}\right)
_{E}\right)  \left(  b_{2}-\left(  b_{2}\right)  _{E}\right)  f_{2}\right)
\right \Vert _{L_{q}\left(  E\right)  }\\
& \lesssim \Vert b_{1}\Vert_{LC_{p_{1},\lambda_{1},P}^{\left \{  x_{0}\right \}
}}\Vert b_{2}\Vert_{LC_{p_{2},\lambda_{2},P}^{\left \{  x_{0}\right \}  }%
}r^{\frac{\gamma}{q}}\int \limits_{2kr}^{\infty}\left(  1+\ln \frac{t}%
{r}\right)  ^{2}\frac{\Vert f\Vert_{L_{p}(E(x_{0},t))}}{t^{\frac{\gamma}%
{p}-\gamma \left(  \lambda_{1}+\lambda_{2}\right)  +1}}dt.
\end{align*}
So, combining all the estimates for $G_{1},$ $G_{2}$, $G_{3}$, $G_{4}$, we get%
\[
G=\left \Vert [\left(  b_{1},b_{2}\right)  ,T_{\Omega}^{P}]f_{2}\right \Vert
_{L_{q}\left(  E\right)  }\lesssim \Vert b_{1}\Vert_{LC_{p_{1},\lambda_{1}%
,P}^{\left \{  x_{0}\right \}  }}\Vert b_{2}\Vert_{LC_{p_{2},\lambda_{2}%
,P}^{\left \{  x_{0}\right \}  }}r^{\frac{\gamma}{q}}\int \limits_{2kr}^{\infty
}\left(  1+\ln \frac{t}{r}\right)  ^{2}\frac{\Vert f\Vert_{L_{p}(E(x_{0},t))}%
}{t^{\frac{\gamma}{p}-\gamma \left(  \lambda_{1}+\lambda_{2}\right)  +1}}dt.
\]
Thus, putting estimates $F$ and $G$ together, we get the desired conclusion%
\[
\left \Vert \lbrack \left(  b_{1},b_{2}\right)  ,T_{\Omega}^{P}]f\right \Vert
_{L_{q}\left(  E\right)  }\lesssim \Vert b_{1}\Vert_{LC_{p_{1},\lambda_{1}%
,P}^{\left \{  x_{0}\right \}  }}\Vert b_{2}\Vert_{LC_{p_{2},\lambda_{2}%
,P}^{\left \{  x_{0}\right \}  }}r^{\frac{\gamma}{q}}\int \limits_{2kr}^{\infty
}\left(  1+\ln \frac{t}{r}\right)  ^{2}\frac{\Vert f\Vert_{L_{p}(E(x_{0},t))}%
}{t^{\frac{\gamma}{p}-\gamma \left(  \lambda_{1}+\lambda_{2}\right)  +1}}dt.
\]
For the case of $p<s$, we can also use the same method, so we omit the
details. This completes the proof of Lemma \ref{lemma2}.
\end{proof}

\begin{lemma}
\label{lemma3}Suppose that $x_{0}\in{\mathbb{R}^{n}}$, $\Omega \in
L_{s}(S^{n-1})$, $1<s\leq \infty$, is $A_{t}$-homogeneous of degree zero. Let
$T_{\Omega,\alpha}^{P}$ be a parabolic linear operator satisfying condition
(\ref{e2}). Let also $0<\alpha<\gamma$ and $1<q,q_{1},p_{i},p<\frac{\gamma
}{\alpha}$ with $\frac{1}{q}=\sum \limits_{i=1}^{m}\frac{1}{p_{i}}+\frac{1}{p}%
$, $\frac{1}{q_{1}}=\frac{1}{q}-\frac{\alpha}{\gamma}$ and $\overrightarrow
{b}\in LC_{p_{i},\lambda_{i},P}^{\left \{  x_{0}\right \}  }({\mathbb{R}^{n}})$
for $0\leq \lambda_{i}<\frac{1}{\gamma}$, $i=1,\ldots,m$.

Then, for $s^{\prime}\leq q$ the inequality
\begin{equation}
\Vert \lbrack \overrightarrow{b},T_{\Omega,\alpha}^{P}]f\Vert_{L_{q_{1}}%
(E(x_{0},r))}\lesssim%
{\displaystyle \prod \limits_{i=1}^{m}}
\Vert \overrightarrow{b}\Vert_{LC_{p_{i},\lambda_{i},P}^{\left \{
x_{0}\right \}  }}r^{\frac{\gamma}{q_{1}}}%
{\displaystyle \int \limits_{2kr}^{\infty}}
\left(  1+\ln \frac{t}{r}\right)  ^{m}\frac{\Vert f\Vert_{L_{p}(E(x_{0},t))}%
}{t^{\gamma \left(  \frac{1}{q_{1}}-\left(
{\displaystyle \sum \limits_{i=1}^{m}}
\lambda_{i}+%
{\displaystyle \sum \limits_{i=1}^{m}}
\frac{1}{p_{i}}\right)  \right)  +1}}dt\label{200*}%
\end{equation}
holds for any ellipsoid $E\left(  x_{0},r\right)  $ and for all $f\in
L_{p}^{loc}({\mathbb{R}^{n}})$. Also, for $q_{1}<s$ the inequality%
\[
\Vert \lbrack \overrightarrow{b},T_{\Omega,\alpha}^{P}]f\Vert_{L_{q_{1}}%
(E(x_{0},r))}\lesssim%
{\displaystyle \prod \limits_{i=1}^{m}}
\Vert \overrightarrow{b}\Vert_{LC_{p_{i},\lambda_{i},P}^{\left \{
x_{0}\right \}  }}\,r^{\frac{\gamma}{q_{1}}-\frac{\gamma}{s}}%
{\displaystyle \int \limits_{2kr}^{\infty}}
\left(  1+\ln \frac{t}{r}\right)  ^{m}\frac{\Vert f\Vert_{L_{p}(E(x_{0},t))}%
}{t^{\gamma \left(  \frac{1}{q_{1}}-\left(  \frac{1}{s}+%
{\displaystyle \sum \limits_{i=1}^{m}}
\lambda_{i}+%
{\displaystyle \sum \limits_{i=1}^{m}}
\frac{1}{p_{i}}\right)  \right)  +1}}dt
\]
holds for any ellipsoid $E\left(  x_{0},r\right)  $ and for all $f\in
L_{p}^{loc}({\mathbb{R}^{n}})$.
\end{lemma}

\begin{proof}
Similar to the proof of Lemma \ref{lemma2}, it is sufficient to show that the
conclusion holds for $m=2$. Let $0<\alpha<\gamma$ and $1<q,q_{1},p_{i}%
,p<\frac{\gamma}{\alpha}$ with $\frac{1}{q}=\sum \limits_{i=1}^{m}\frac
{1}{p_{i}}+\frac{1}{p}$, $\frac{1}{q_{1}}=\frac{1}{q}-\frac{\alpha}{\gamma}$
and $\overrightarrow{b}\in LC_{p_{i},\lambda_{i},P}^{\left \{  x_{0}\right \}
}({\mathbb{R}^{n}})$ for $0\leq \lambda_{i}<\frac{1}{\gamma}$, $i=1,\ldots,m$.
As in the proof of Lemma \ref{lemma2}, we split $f=f_{1}+f_{2}$ in form
(\ref{9}) and have%
\[
\left \Vert \lbrack \left(  b_{1},b_{2}\right)  ,T_{\Omega,\alpha}%
^{P}]f\right \Vert _{L_{q_{1}}\left(  E\right)  }\leq \left \Vert \lbrack \left(
b_{1},b_{2}\right)  ,T_{\Omega,\alpha}^{P}]f_{1}\right \Vert _{L_{q_{1}}\left(
E\right)  }+\left \Vert [\left(  b_{1},b_{2}\right)  ,T_{\Omega,\alpha}%
^{P}]f_{2}\right \Vert _{L_{q_{1}}\left(  E\right)  }=:A+B.
\]
Let us estimate $A+B$, respectively.

For $[\left(  b_{1},b_{2}\right)  ,T_{\Omega,\alpha}^{P}]f_{1}\left(
x\right)  $, we have the following decomposition,%
\begin{align*}
\lbrack \left(  b_{1},b_{2}\right)  ,T_{\Omega,\alpha}^{P}]f_{1}\left(
x\right)   & =\left(  b_{1}\left(  x\right)  -\left(  b_{1}\right)
_{E}\right)  \left(  b_{2}\left(  x\right)  -\left(  b_{2}\right)
_{E}\right)  T_{\Omega,\alpha}^{P}f_{1}\left(  x\right) \\
& -\left(  b_{1}\left(  \cdot \right)  -\left(  b_{1}\right)  _{E}\right)
T_{\Omega,\alpha}^{P}\left(  \left(  b_{2}\left(  \cdot \right)  -\left(
b_{2}\right)  _{E}\right)  f_{1}\right)  \left(  x\right) \\
& +\left(  b_{2}\left(  x\right)  -\left(  b_{2}\right)  _{E}\right)
T_{\Omega,\alpha}^{P}\left(  \left(  b_{1}\left(  x\right)  -\left(
b_{1}\right)  _{E}\right)  f_{1}\right)  \left(  x\right)  -\\
& T_{\Omega,\alpha}^{P}\left(  \left(  b_{1}\left(  \cdot \right)  -\left(
b_{1}\right)  _{E}\right)  \left(  b_{2}\left(  \cdot \right)  -\left(
b_{2}\right)  _{E}\right)  f_{1}\right)  \left(  x\right)  .
\end{align*}
Hence, we get%
\begin{align}
A  & =\left \Vert [\left(  b_{1},b_{2}\right)  ,T_{\Omega,\alpha}^{P}%
]f_{1}\right \Vert _{L_{q_{1}}\left(  E\right)  }\lesssim \nonumber \\
& \left \Vert \left(  b_{1}-\left(  b_{1}\right)  _{E}\right)  \left(
b_{2}\left(  x\right)  -\left(  b_{2}\right)  _{E}\right)  T_{\Omega,\alpha
}^{P}f_{1}\right \Vert _{L_{q_{1}}\left(  E\right)  }\nonumber \\
& +\left \Vert \left(  b_{1}-\left(  b_{1}\right)  _{E}\right)  T_{\Omega
,\alpha}^{P}\left(  \left(  b_{2}-\left(  b_{2}\right)  _{E}\right)
f_{1}\right)  \right \Vert _{L_{q_{1}}\left(  E\right)  }\nonumber \\
& +\left \Vert \left(  b_{2}-\left(  b_{2}\right)  _{E}\right)  T_{\Omega
,\alpha}^{P}\left(  \left(  b_{1}-\left(  b_{1}\right)  _{E}\right)
f_{1}\right)  \right \Vert _{L_{q_{1}}\left(  E\right)  }\nonumber \\
& +\left \Vert T_{\Omega,\alpha}^{P}\left(  \left(  b_{1}-\left(  b_{1}\right)
_{E}\right)  \left(  b_{2}-\left(  b_{2}\right)  _{E}\right)  f_{1}\right)
\right \Vert _{L_{q_{1}}\left(  E\right)  }\nonumber \\
& \equiv A_{1}+A_{2}+A_{3}+A_{4}.\label{0*}%
\end{align}
One observes that the estimate of $A_{2}$ is analogous to that of $A_{3}$.
Thus, we will only estimate $A_{1}$, $A_{2}$ and $A_{4}$.

To estimate $A_{1}$, let $1<\overline{q},\overline{r}<\infty$, such
that$\frac{1}{\overline{q}}=\frac{1}{p}-\frac{\alpha}{\gamma},$ $\frac
{1}{q_{1}}=\frac{1}{\overline{r}}+\frac{1}{\overline{q}},\frac{1}{\overline
{r}}=\frac{1}{p_{1}}+\frac{1}{p_{2}}$. Then, using H\"{o}lder's inequality and
by the boundedness of $T_{\Omega,\alpha}^{P}$ from $L_{p}$ into $L_{\overline
{q}}$ (see Theorem 0.1 in \cite{Gurbuz3}) and by (\ref{1}) it follows that:%
\begin{align*}
A_{1}  & =\left \Vert \left(  b_{1}-\left(  b_{1}\right)  _{E}\right)  \left(
b_{2}\left(  x\right)  -\left(  b_{2}\right)  _{E}\right)  T_{\Omega,\alpha
}^{P}f_{1}\right \Vert _{L_{q_{1}}\left(  E\right)  }\\
& \lesssim \left \Vert \left(  b_{1}-\left(  b_{1}\right)  _{E}\right)  \left(
b_{2}-\left(  b_{2}\right)  _{E}\right)  \right \Vert _{L_{\overline{r}}\left(
E\right)  }\left \Vert T_{\Omega,\alpha}^{P}f_{1}\right \Vert _{L_{\overline{q}%
}\left(  E\right)  }\\
& \lesssim \left \Vert b_{1}-\left(  b_{1}\right)  _{B}\right \Vert _{L_{p_{1}%
}\left(  E\right)  }\left \Vert b_{2}-\left(  b_{2}\right)  _{B}\right \Vert
_{L_{p_{2}}\left(  E\right)  }\left \Vert f\right \Vert _{L_{p}\left(
2kE\right)  }\\
& \lesssim \left \Vert b_{1}-\left(  b_{1}\right)  _{E}\right \Vert _{L_{p_{1}%
}\left(  E\right)  }\left \Vert b_{2}-\left(  b_{2}\right)  _{E}\right \Vert
_{L_{p_{2}}\left(  E\right)  }r^{\gamma \left(  \frac{1}{p}-\frac{\alpha
}{\gamma}\right)  }\int \limits_{2kr}^{\infty}\Vert f\Vert_{L_{p}(E(x_{0}%
,t))}\frac{dt}{t^{\frac{\gamma}{p}+1-\alpha}}\\
& \lesssim \Vert b_{1}\Vert_{LC_{p_{1},\lambda_{1},P}^{\left \{  x_{0}\right \}
}}\Vert b_{2}\Vert_{LC_{p_{2},\lambda_{2},P}^{\left \{  x_{0}\right \}  }%
}r^{\gamma \left(  \frac{1}{p_{1}}+\frac{1}{p_{2}}+\frac{1}{p}-\frac{\alpha
}{\gamma}\right)  }\int \limits_{2kr}^{\infty}\left(  1+\ln \frac{t}{r}\right)
^{2}t^{-\frac{\gamma}{p}+\gamma \left(  \lambda_{1}+\lambda_{2}\right)
-1+\alpha}\Vert f\Vert_{L_{p}(E(x_{0},t))}dt\\
& \lesssim \Vert b_{1}\Vert_{LC_{p_{1},\lambda_{1},P}^{\left \{  x_{0}\right \}
}}\Vert b_{2}\Vert_{LC_{p_{2},\lambda_{2},P}^{\left \{  x_{0}\right \}  }%
}r^{\frac{\gamma}{q_{1}}}\int \limits_{2kr}^{\infty}\left(  1+\ln \frac{t}%
{r}\right)  ^{2}\frac{\Vert f\Vert_{L_{p}(E(x_{0},t))}}{t^{\frac{\gamma}%
{q_{1}}-\gamma \left(  \lambda_{1}+\lambda_{2}\right)  -\gamma \left(  \frac
{1}{p_{1}}+\frac{1}{p_{2}}\right)  +1}}dt.
\end{align*}
To estimate $A_{2}$, let $1<\tau<\infty$, such that $\frac{1}{q_{1}}=\frac
{1}{p_{1}}+\frac{1}{\tau}$. Then, similar to the estimates for $A_{1}$, we
have%
\begin{align*}
A_{2} &  =\left \Vert \left(  b_{1}-\left(  b_{1}\right)  _{E}\right)
T_{\Omega,\alpha}^{P}\left(  \left(  b_{2}-\left(  b_{2}\right)  _{E}\right)
f_{1}\right)  \right \Vert _{L_{q_{1}}\left(  E\right)  }\\
&  \lesssim \left \Vert b_{1}-\left(  b_{1}\right)  _{E}\right \Vert _{L_{p_{1}%
}\left(  E\right)  }\left \Vert T_{\Omega,\alpha}^{P}\left(  \left(
b_{2}\left(  \cdot \right)  -\left(  b_{2}\right)  _{E}\right)  f_{1}\right)
\right \Vert _{L_{\tau}\left(  E\right)  }\\
&  \lesssim \left \Vert b_{1}-\left(  b_{1}\right)  _{E}\right \Vert _{L_{p_{1}%
}\left(  E\right)  }\left \Vert \left(  b_{2}\left(  \cdot \right)  -\left(
b_{2}\right)  _{E}\right)  f_{1}\right \Vert _{L_{k}\left(  E\right)  }\\
&  \lesssim \left \Vert b_{1}-\left(  b_{1}\right)  _{E}\right \Vert _{L_{p_{1}%
}\left(  E\right)  }\left \Vert b_{2}-\left(  b_{2}\right)  _{E}\right \Vert
_{L_{p_{2}}\left(  2kE\right)  }\left \Vert f\right \Vert _{L_{p}\left(
2kE\right)  },
\end{align*}
where $1<k<\frac{2\gamma}{\alpha}$, such that $\frac{1}{k}=\frac{1}{p_{2}%
}+\frac{1}{p}=\frac{1}{\tau}+\frac{\alpha}{\gamma}$. By (\ref{1}) and
(\ref{2}), we get%
\[
A_{2}\lesssim \Vert b_{1}\Vert_{LC_{p_{1},\lambda_{1},P}^{\left \{
x_{0}\right \}  }}\Vert b_{2}\Vert_{LC_{p_{2},\lambda_{2},P}^{\left \{
x_{0}\right \}  }}r^{\frac{\gamma}{q_{1}}}\int \limits_{2kr}^{\infty}\left(
1+\ln \frac{t}{r}\right)  ^{2}\frac{\Vert f\Vert_{L_{p}(E(x_{0},t))}}%
{t^{\frac{\gamma}{q_{1}}-\gamma \left(  \lambda_{1}+\lambda_{2}\right)
-\gamma \left(  \frac{1}{p_{1}}+\frac{1}{p_{2}}\right)  +1}}dt.
\]
In a similar way, $A_{3}$ has the same estimate as above, so we omit the
details. Then we have that%
\[
A_{3}\lesssim \Vert b_{1}\Vert_{LC_{p_{1},\lambda_{1},P}^{\left \{
x_{0}\right \}  }}\Vert b_{2}\Vert_{LC_{p_{2},\lambda_{2},P}^{\left \{
x_{0}\right \}  }}r^{\frac{\gamma}{q_{1}}}\int \limits_{2kr}^{\infty}\left(
1+\ln \frac{t}{r}\right)  ^{2}\frac{\Vert f\Vert_{L_{p}(E(x_{0},t))}}%
{t^{\frac{\gamma}{q_{1}}-\gamma \left(  \lambda_{1}+\lambda_{2}\right)
-\gamma \left(  \frac{1}{p_{1}}+\frac{1}{p_{2}}\right)  +1}}dt.
\]
Now let us consider the term $A_{4}$. Let $1<q,\tau<\frac{2\gamma}{\alpha} $,
such that $\frac{1}{q}=\frac{1}{\tau}+\frac{1}{p}$, $\frac{1}{\tau}=\frac
{1}{p_{1}}+\frac{1}{p_{2}}$ and $\frac{1}{q_{1}}=\frac{1}{q}-\frac{\alpha
}{\gamma}$. Then by the boundedness of $T_{\Omega,\alpha}^{P}$ from $L_{q}$
into $L_{q_{1}}$ (see Theorem 0.1 in \cite{Gurbuz3}), H\"{o}lder's inequality
and (\ref{2}), we obtain%
\begin{align*}
A_{4}  & =\left \Vert T_{\Omega,\alpha}^{P}\left(  \left(  b_{1}-\left(
b_{1}\right)  _{E}\right)  \left(  b_{2}-\left(  b_{2}\right)  _{E}\right)
f_{1}\right)  \right \Vert _{L_{q_{1}}\left(  E\right)  }\\
& \lesssim \left \Vert \left(  b_{1}-\left(  b_{1}\right)  _{E}\right)  \left(
b_{2}-\left(  b_{2}\right)  _{E}\right)  f_{1}\right \Vert _{L_{q}\left(
E\right)  }\\
& \lesssim \left \Vert \left(  b_{1}-\left(  b_{1}\right)  _{E}\right)  \left(
b_{2}-\left(  b_{2}\right)  _{E}\right)  \right \Vert _{L_{\tau}\left(
E\right)  }\left \Vert f_{1}\right \Vert _{L_{p}\left(  E\right)  }\\
& \lesssim \left \Vert b_{1}-\left(  b_{1}\right)  _{E}\right \Vert _{L_{p_{1}%
}\left(  2kE\right)  }\left \Vert b_{2}-\left(  b_{2}\right)  _{E}\right \Vert
_{L_{p_{2}}\left(  2kE\right)  }\left \Vert f\right \Vert _{L_{p}\left(
2kE\right)  }\\
& \lesssim \Vert b_{1}\Vert_{LC_{p_{1},\lambda_{1},P}^{\left \{  x_{0}\right \}
}}\Vert b_{2}\Vert_{LC_{p_{2},\lambda_{2},P}^{\left \{  x_{0}\right \}  }%
}r^{\frac{\gamma}{q_{1}}}\int \limits_{2kr}^{\infty}\left(  1+\ln \frac{t}%
{r}\right)  ^{2}\frac{\Vert f\Vert_{L_{p}(E(x_{0},t))}}{t^{\frac{\gamma}%
{q_{1}}-\gamma \left(  \lambda_{1}+\lambda_{2}\right)  -\gamma \left(  \frac
{1}{p_{1}}+\frac{1}{p_{2}}\right)  +1}}dt.
\end{align*}
Combining all the estimates of $A_{1}$, $A_{2}$, $A_{3}$, $A_{4}$; we get%
\[
A=\left \Vert [\left(  b_{1},b_{2}\right)  ,T_{\Omega,\alpha}^{P}%
]f_{1}\right \Vert _{L_{q_{1}}\left(  E\right)  }\lesssim \Vert b_{1}%
\Vert_{LC_{p_{1},\lambda_{1},P}^{\left \{  x_{0}\right \}  }}\Vert b_{2}%
\Vert_{LC_{p_{2},\lambda_{2},P}^{\left \{  x_{0}\right \}  }}r^{\frac{\gamma
}{q_{1}}}\int \limits_{2kr}^{\infty}\left(  1+\ln \frac{t}{r}\right)  ^{2}%
\frac{\Vert f\Vert_{L_{p}(E(x_{0},t))}}{t^{\frac{\gamma}{q_{1}}-\gamma \left(
\lambda_{1}+\lambda_{2}\right)  -\gamma \left(  \frac{1}{p_{1}}+\frac{1}{p_{2}%
}\right)  +1}}dt.
\]
Now, let us estimate $B=\left \Vert [\left(  b_{1},b_{2}\right)  ,T_{\Omega
,\alpha}^{P}]f_{2}\right \Vert _{L_{q_{1}}\left(  E\right)  }$. For $B$, it's
similar to (\ref{0*}) we also write%
\begin{align*}
B  & =\left \Vert [\left(  b_{1},b_{2}\right)  ,T_{\Omega,\alpha}^{P}%
]f_{2}\right \Vert _{L_{q_{1}}\left(  E\right)  }\lesssim \\
& \left \Vert \left(  b_{1}-\left(  b_{1}\right)  _{E}\right)  \left(
b_{2}\left(  x\right)  -\left(  b_{2}\right)  _{E}\right)  T_{\Omega,\alpha
}^{P}f_{2}\right \Vert _{L_{q_{1}}\left(  E\right)  }\\
& +\left \Vert \left(  b_{1}-\left(  b_{1}\right)  _{E}\right)  T_{\Omega
,\alpha}^{P}\left(  \left(  b_{2}-\left(  b_{2}\right)  _{E}\right)
f_{2}\right)  \right \Vert _{L_{q_{1}}\left(  E\right)  }\\
& +\left \Vert \left(  b_{2}-\left(  b_{2}\right)  _{E}\right)  T_{\Omega
,\alpha}^{P}\left(  \left(  b_{1}-\left(  b_{1}\right)  _{E}\right)
f_{2}\right)  \right \Vert _{L_{q_{1}}\left(  E\right)  }\\
& +\left \Vert T_{\Omega,\alpha}^{P}\left(  \left(  b_{1}-\left(  b_{1}\right)
_{E}\right)  \left(  b_{2}-\left(  b_{2}\right)  _{E}\right)  f_{2}\right)
\right \Vert _{L_{q_{1}}\left(  E\right)  }\\
& \equiv B_{1}+B_{2}+B_{3}+B_{4}.
\end{align*}
To estimate $B_{1}$, let $1<p_{1},p_{2}<\frac{2\gamma}{\alpha}$, such that
$\frac{1}{q_{1}}=\frac{1}{\overline{p}}+\frac{1}{\overline{q}}$, $\frac
{1}{\overline{p}}=\frac{1}{p_{1}}+\frac{1}{p_{2}}$ and $\frac{1}{\overline{q}%
}=\frac{1}{p}-\frac{\alpha}{\gamma}$. Then, using H\"{o}lder's inequality and
by (\ref{1}), we have%
\begin{align*}
B_{1}  & =\left \Vert \left(  b_{1}-\left(  b_{1}\right)  _{E}\right)  \left(
b_{2}\left(  x\right)  -\left(  b_{2}\right)  _{E}\right)  T_{\Omega,\alpha
}^{P}f_{2}\right \Vert _{L_{q_{1}}\left(  E\right)  }\\
& \lesssim \left \Vert \left(  b_{1}-\left(  b_{1}\right)  _{E}\right)  \left(
b_{2}-\left(  b_{2}\right)  _{E}\right)  \right \Vert _{L_{\overline{p}}\left(
E\right)  }\left \Vert T_{\Omega,\alpha}^{P}f_{2}\right \Vert _{L_{\overline{q}%
}\left(  E\right)  }\\
& \lesssim \left \Vert b_{1}-\left(  b_{1}\right)  _{E}\right \Vert _{L_{p_{1}%
}\left(  E\right)  }\left \Vert b_{2}-\left(  b_{2}\right)  _{E}\right \Vert
_{L_{p_{2}}\left(  E\right)  }r^{\frac{\gamma}{_{\overline{q}}}}%
\int \limits_{2kr}^{\infty}\Vert f\Vert_{L_{p}(E(x_{0},t))}t^{-\frac{\gamma
}{\overline{q}}-1}dt\\
& \lesssim \Vert b_{1}\Vert_{LC_{p_{1},\lambda_{1},P}^{\left \{  x_{0}\right \}
}}\Vert b_{2}\Vert_{LC_{p_{2},\lambda_{2},P}^{\left \{  x_{0}\right \}  }%
}r^{\gamma \left(  \frac{1}{p_{1}}+\frac{1}{p_{2}}+\frac{1}{p}-\frac{\alpha
}{\gamma}\right)  }\\
& \times \int \limits_{2kr}^{\infty}\left(  1+\ln \frac{t}{r}\right)
^{2}t^{-\frac{\gamma}{p}+\gamma \left(  \lambda_{1}+\lambda_{2}\right)
-1+\alpha}\Vert f\Vert_{L_{p}(E(x_{0},t))}dt\\
& \lesssim \Vert b_{1}\Vert_{LC_{p_{1},\lambda_{1},P}^{\left \{  x_{0}\right \}
}}\Vert b_{2}\Vert_{LC_{p_{2},\lambda_{2},P}^{\left \{  x_{0}\right \}  }%
}r^{\frac{\gamma}{q_{1}}}\int \limits_{2kr}^{\infty}\left(  1+\ln \frac{t}%
{r}\right)  ^{2}\frac{\Vert f\Vert_{L_{p}(E(x_{0},t))}}{t^{\frac{\gamma}%
{q_{1}}-\gamma \left(  \lambda_{1}+\lambda_{2}\right)  -\gamma \left(  \frac
{1}{p_{1}}+\frac{1}{p_{2}}\right)  +1}}dt,
\end{align*}
where in the second inequality we have used the following fact:

When $s^{\prime}\leq \overline{q}$, by (\ref{7}), Fubini's theorem,
H\"{o}lder's inequality and (\ref{e311}), we have%
\begin{align*}
\left \vert T_{\Omega,\alpha}^{P}f_{2}\left(  x\right)  \right \vert  &  \leq
c_{0}\int \limits_{\left(  2kE\right)  ^{C}}\left \vert \Omega \left(
x-y\right)  \right \vert \frac{\left \vert f\left(  y\right)  \right \vert }%
{\rho \left(  x_{0}-y\right)  ^{\gamma-\alpha}}dy\\
&  \approx \int \limits_{2kr}^{\infty}\int \limits_{2kr<\rho \left(
x_{0}-y\right)  <t}\left \vert \Omega \left(  x-y\right)  \right \vert \left \vert
f\left(  y\right)  \right \vert dyt^{-1-\gamma+\alpha}dt\\
&  \lesssim \int \limits_{2kr}^{\infty}\int \limits_{E\left(  x_{0},t\right)
}\left \vert \Omega \left(  x-y\right)  \right \vert \left \vert f\left(
y\right)  \right \vert dyt^{-1-\gamma+\alpha}dt\\
&  \lesssim \int \limits_{2kr}^{\infty}\left \Vert f\right \Vert _{L_{p}\left(
E\left(  x_{0},t\right)  \right)  }\left \Vert \Omega \left(  x-\cdot \right)
\right \Vert _{L_{s}\left(  E\left(  x_{0},t\right)  \right)  }\left \vert
E\left(  x_{0},t\right)  \right \vert ^{1-\frac{1}{p}-\frac{1}{s}}%
t^{-1-\gamma+\alpha}dt\\
&  \lesssim \int \limits_{2kr}^{\infty}\left \Vert f\right \Vert _{L_{p}\left(
E\left(  x_{0},t\right)  \right)  }t^{-1-\frac{\gamma}{\overline{q}}}dt.
\end{align*}
Moreover, for all $p\in \left[  1,\infty \right)  $ the inequality%
\[
\left \Vert T_{\Omega,\alpha}^{P}f_{2}\right \Vert _{L_{p}\left(  E\right)
}\lesssim r^{\frac{\gamma}{\overline{q}}}\int \limits_{2kr}^{\infty}\left \Vert
f\right \Vert _{L_{p}\left(  E\left(  x_{0},t\right)  \right)  }\frac
{dt}{t^{\frac{\gamma}{\overline{q}}+1}}%
\]
is valid.

On the other hand, for the estimates used in $B_{2}$, $B_{3}$, we have to
prove the below inequality:%
\begin{equation}
\left \vert T_{\Omega,\alpha}^{P}\left(  \left(  b_{2}-\left(  b_{2}\right)
_{E}\right)  f_{2}\right)  \left(  x\right)  \right \vert \lesssim \left \Vert
b_{2}\right \Vert _{LC_{p_{2},\lambda_{2},P}^{\left \{  x_{0}\right \}  }}%
{\displaystyle \int \limits_{2kr}^{\infty}}
\left(  1+\ln \frac{t}{r}\right)  t^{-\frac{\gamma}{p}+\gamma \lambda
_{2}-1+\alpha}\left \Vert f\right \Vert _{L_{p}\left(  E\left(  x_{0},t\right)
\right)  }dt.\label{11*}%
\end{equation}
Indeed, when $s^{\prime}\leq q$, for $x\in E$, by Fubini's theorem and
applying H\"{o}lder's inequality and from (\ref{b}), (\ref{c}), (\ref{e311})
we have

$\left \vert T_{\Omega,\alpha}^{P}\left(  \left(  b_{2}\left(  \cdot \right)
-\left(  b_{2}\right)  _{E}\right)  f_{2}\right)  \left(  x\right)
\right \vert \lesssim%
{\displaystyle \int \limits_{\left(  2kE\right)  ^{C}}}
\left \vert b_{2}\left(  y\right)  -\left(  b_{2}\right)  _{E}\right \vert
\left \vert \Omega \left(  x-y\right)  \right \vert \frac{\left \vert f\left(
y\right)  \right \vert }{\rho \left(  x-y\right)  ^{\gamma-\alpha}}dy$

$\lesssim%
{\displaystyle \int \limits_{\left(  2kE\right)  ^{C}}}
\left \vert b_{2}\left(  y\right)  -\left(  b_{2}\right)  _{E}\right \vert
\left \vert \Omega \left(  x-y\right)  \right \vert \frac{\left \vert f\left(
y\right)  \right \vert }{\rho \left(  x_{0}-y\right)  ^{\gamma-\alpha}}dy$

$\approx%
{\displaystyle \int \limits_{2kr}^{\infty}}
{\displaystyle \int \limits_{2kr<\rho \left(  x_{0}-y\right)  <t}}
\left \vert b_{2}\left(  y\right)  -\left(  b_{2}\right)  _{E}\right \vert
\left \vert \Omega \left(  x-y\right)  \right \vert \left \vert f\left(  y\right)
\right \vert dy\frac{dt}{t^{\gamma-\alpha+1}} $

$\lesssim%
{\displaystyle \int \limits_{2kr}^{\infty}}
{\displaystyle \int \limits_{E\left(  x_{0},t\right)  }}
\left \vert b_{2}\left(  y\right)  -\left(  b_{2}\right)  _{E\left(
x_{0},t\right)  }\right \vert \left \vert \Omega \left(  x-y\right)  \right \vert
\left \vert f\left(  y\right)  \right \vert dy\frac{dt}{t^{\gamma-\alpha+1}}$

$+%
{\displaystyle \int \limits_{2kr}^{\infty}}
\left \vert \left(  b_{2}\right)  _{E\left(  x_{0},r\right)  }-\left(
b_{2}\right)  _{E\left(  x_{0},t\right)  }\right \vert
{\displaystyle \int \limits_{E\left(  x_{0},t\right)  }}
\left \vert \Omega \left(  x-y\right)  \right \vert \left \vert f\left(  y\right)
\right \vert dy\frac{dt}{t^{\gamma-\alpha+1}}$

$\lesssim%
{\displaystyle \int \limits_{2kr}^{\infty}}
\left \Vert b_{2}\left(  \cdot \right)  -\left(  b_{2}\right)  _{E\left(
x_{0},t\right)  }\right \Vert _{L_{p_{2}}\left(  E\left(  x_{0},t\right)
\right)  }\left \Vert \Omega \left(  \cdot-y\right)  \right \Vert _{L_{s}\left(
E\left(  x_{0},t\right)  \right)  }\left \Vert f\right \Vert _{L_{p}\left(
E\left(  x_{0},t\right)  \right)  }\left \vert E\left(  x_{0},t\right)
\right \vert ^{1-\frac{1}{p_{2}}-\frac{1}{s}-\frac{1}{p}}\frac{dt}%
{t^{\gamma-\alpha+1}}$

$+%
{\displaystyle \int \limits_{2kr}^{\infty}}
\left \vert \left(  b_{2}\right)  _{E\left(  x_{0},r\right)  }-\left(
b_{2}\right)  _{E\left(  x_{0},t\right)  }\right \vert \left \Vert f\right \Vert
_{L_{p}\left(  E\left(  x_{0},t\right)  \right)  }\left \Vert \Omega \left(
\cdot-y\right)  \right \Vert _{L_{s}\left(  E\left(  x_{0},t\right)  \right)
}\left \vert E\left(  x_{0},t\right)  \right \vert ^{1-\frac{1}{p}-\frac{1}{s}%
}\frac{dt}{t^{\gamma-\alpha+1}}$

$\lesssim%
{\displaystyle \int \limits_{2kr}^{\infty}}
\left \Vert b_{2}\left(  \cdot \right)  -\left(  b_{2}\right)  _{E\left(
x_{0},t\right)  }\right \Vert _{L_{p_{2}}\left(  E\left(  x_{0},t\right)
\right)  }\left \Vert f\right \Vert _{L_{p}\left(  E\left(  x_{0},t\right)
\right)  }t^{-1-\gamma \left(  \frac{1}{p_{2}}+\frac{1}{p}\right)  +\alpha}dt$

$+\left \Vert b_{2}\right \Vert _{LC_{p_{2},\lambda_{2},P}^{\left \{
x_{0}\right \}  }}%
{\displaystyle \int \limits_{2kr}^{\infty}}
\left(  1+\ln \frac{t}{r}\right)  \left \Vert f\right \Vert _{L_{p}\left(
E\left(  x_{0},t\right)  \right)  }t^{-1-\frac{\gamma}{p}+\gamma \lambda
_{2}+\alpha}dt$

$\lesssim \left \Vert b_{2}\right \Vert _{LC_{p_{2},\lambda_{2},P}^{\left \{
x_{0}\right \}  }}%
{\displaystyle \int \limits_{2kr}^{\infty}}
\left(  1+\ln \frac{t}{r}\right)  t^{-\frac{\gamma}{p}+\gamma \lambda
_{2}-1+\alpha}\left \Vert f\right \Vert _{L_{p}\left(  E\left(  x_{0},t\right)
\right)  }dt.$

This completes the proof of inequality (\ref{11*}).

Let $1<\tau<\infty$, such that $\frac{1}{q_{1}}=\frac{1}{p_{1}}+\frac{1}{\tau
}$ and $\frac{1}{\tau}=\frac{1}{p_{2}}+\frac{1}{p}-\frac{\alpha}{\gamma}$.
Then, using H\"{o}lder's inequality and from (\ref{11*}) and (\ref{c}), we get%
\begin{align*}
B_{2} &  =\left \Vert \left(  b_{1}-\left(  b_{1}\right)  _{E}\right)
T_{\Omega,\alpha}^{P}\left(  \left(  b_{2}-\left(  b_{2}\right)  _{E}\right)
f_{2}\right)  \right \Vert _{L_{q_{1}}\left(  E\right)  }\\
&  \lesssim \left \Vert b_{1}-\left(  b_{1}\right)  _{E}\right \Vert _{L_{p_{1}%
}\left(  E\right)  }\left \Vert T_{\Omega,\alpha}^{P}\left(  \left(
b_{2}\left(  \cdot \right)  -\left(  b_{2}\right)  _{E}\right)  f_{2}\right)
\right \Vert _{L_{\tau}\left(  E\right)  }\\
&  \lesssim \left \Vert b_{1}-\left(  b_{1}\right)  _{E}\right \Vert _{L_{p_{1}%
}\left(  E\right)  }\left \Vert b_{2}\right \Vert _{LC_{p_{2},\lambda_{2}%
,P}^{\left \{  x_{0}\right \}  }}r^{\frac{\gamma}{\tau}}%
{\displaystyle \int \limits_{2kr}^{\infty}}
\left(  1+\ln \frac{t}{r}\right)  t^{-\frac{\gamma}{p}+\gamma \lambda
_{2}-1+\alpha}\left \Vert f\right \Vert _{L_{p}\left(  E\left(  x_{0},t\right)
\right)  }dt\\
&  \lesssim \Vert b_{1}\Vert_{LC_{p_{1},\lambda_{1},P}^{\left \{  x_{0}\right \}
}}\Vert b_{2}\Vert_{LC_{p_{2},\lambda_{2},P}^{\left \{  x_{0}\right \}  }%
}r^{\frac{\gamma}{q_{1}}}\int \limits_{2kr}^{\infty}\left(  1+\ln \frac{t}%
{r}\right)  ^{2}\frac{\Vert f\Vert_{L_{p}(E(x_{0},t))}}{t^{\frac{\gamma}%
{q_{1}}-\gamma \left(  \lambda_{1}+\lambda_{2}\right)  -\gamma \left(  \frac
{1}{p_{1}}+\frac{1}{p_{2}}\right)  +1}}dt.
\end{align*}
Similarly, $B_{3}$ has the same estimate above, so here we omit the details.
Then the inequality%
\begin{align*}
B_{3}  & =\left \Vert \left(  b_{2}-\left(  b_{2}\right)  _{E}\right)
T_{\Omega,\alpha}^{P}\left(  \left(  b_{1}-\left(  b_{1}\right)  _{E}\right)
f_{2}\right)  \right \Vert _{L_{q_{1}}\left(  E\right)  }\\
& \lesssim \Vert b_{1}\Vert_{LC_{p_{1},\lambda_{1},P}^{\left \{  x_{0}\right \}
}}\Vert b_{2}\Vert_{LC_{p_{2},\lambda_{2},P}^{\left \{  x_{0}\right \}  }%
}r^{\frac{\gamma}{q_{1}}}\int \limits_{2kr}^{\infty}\left(  1+\ln \frac{t}%
{r}\right)  ^{2}\frac{\Vert f\Vert_{L_{p}(E(x_{0},t))}}{t^{\frac{\gamma}%
{q_{1}}-\gamma \left(  \lambda_{1}+\lambda_{2}\right)  -\gamma \left(  \frac
{1}{p_{1}}+\frac{1}{p_{2}}\right)  +1}}dt
\end{align*}
is valid.

Now, let us estimate $B_{4}=\left \Vert T_{\Omega,\alpha}^{P}\left(  \left(
b_{1}-\left(  b_{1}\right)  _{E}\right)  \left(  b_{2}-\left(  b_{2}\right)
_{E}\right)  f_{2}\right)  \right \Vert _{L_{q_{1}}\left(  E\right)  }$. It's
similar to the estimate of (\ref{11*}), for any $x\in E$, we also write

$\left \vert T_{\Omega,\alpha}^{P}\left(  \left(  b_{1}-\left(  b_{1}\right)
_{E}\right)  \left(  b_{2}-\left(  b_{2}\right)  _{E}\right)  f_{2}\right)
\left(  x\right)  \right \vert $

$\lesssim%
{\displaystyle \int \limits_{2kr}^{\infty}}
{\displaystyle \int \limits_{E\left(  x_{0},t\right)  }}
\left \vert b_{1}\left(  y\right)  -\left(  b_{1}\right)  _{E\left(
x_{0},t\right)  }\right \vert \left \vert b_{2}\left(  y\right)  -\left(
b_{2}\right)  _{E\left(  x_{0},t\right)  }\right \vert \left \vert
\Omega(x-y)\right \vert \left \vert f\left(  y\right)  \right \vert dy\frac
{dt}{t^{\gamma-\alpha+1}}$

$+%
{\displaystyle \int \limits_{2kr}^{\infty}}
{\displaystyle \int \limits_{E\left(  x_{0},t\right)  }}
\left \vert b_{1}\left(  y\right)  -\left(  b_{1}\right)  _{E\left(
x_{0},t\right)  }\right \vert \left \vert \left(  b_{2}\right)  _{E\left(
x_{0},t\right)  }-\left(  b_{2}\right)  _{E\left(  x_{0},r\right)
}\right \vert \left \vert \Omega(x-y)\right \vert \left \vert f\left(  y\right)
\right \vert dy\frac{dt}{t^{\gamma-\alpha+1}}$

$+%
{\displaystyle \int \limits_{2kr}^{\infty}}
{\displaystyle \int \limits_{E\left(  x_{0},t\right)  }}
\left \vert \left(  b_{1}\right)  _{E\left(  x_{0},t\right)  }-\left(
b_{2}\right)  _{E\left(  x_{0},r\right)  }\right \vert \left \vert b_{2}\left(
y\right)  -\left(  b_{2}\right)  _{E\left(  x_{0},t\right)  }\right \vert
\left \vert \Omega(x-y)\right \vert \left \vert f\left(  y\right)  \right \vert
dy\frac{dt}{t^{\gamma-\alpha+1}}$

$+%
{\displaystyle \int \limits_{2kr}^{\infty}}
{\displaystyle \int \limits_{E\left(  x_{0},t\right)  }}
\left \vert \left(  b_{1}\right)  _{E\left(  x_{0},t\right)  }-\left(
b_{2}\right)  _{E\left(  x_{0},r\right)  }\right \vert \left \vert \left(
b_{2}\right)  _{E\left(  x_{0},t\right)  }-\left(  b_{2}\right)  _{E\left(
x_{0},r\right)  }\right \vert \left \vert \Omega(x-y)\right \vert \left \vert
f\left(  y\right)  \right \vert dy\frac{dt}{t^{\gamma-\alpha+1}} $

$\equiv B_{41}+B_{42}+B_{43}+B_{44}.$

Let us estimate $B_{41}$, $B_{42}$, $B_{43}$, $B_{44}$, respectively.

Firstly, to estimate $B_{41}$, similar to the estimate of (\ref{11*}), we get%
\[
B_{41}\lesssim \Vert b_{1}\Vert_{LC_{p_{1},\lambda_{1},P}^{\left \{
x_{0}\right \}  }}\Vert b_{2}\Vert_{LC_{p_{2},\lambda_{2},P}^{\left \{
x_{0}\right \}  }}\int \limits_{2kr}^{\infty}\left(  1+\ln \frac{t}{r}\right)
^{2}\frac{\Vert f\Vert_{L_{p}(E(x_{0},t))}}{t^{\frac{\gamma}{q_{1}}%
-\gamma \left(  \lambda_{1}+\lambda_{2}\right)  -\gamma \left(  \frac{1}{p_{1}%
}+\frac{1}{p_{2}}\right)  +1}}dt.
\]
Secondly, to estimate $B_{42}$ and $B_{43}$, from (\ref{11*}), (\ref{b}) and
(\ref{c}), it follows that%
\[
B_{42}\lesssim \Vert b_{1}\Vert_{LC_{p_{1},\lambda_{1},P}^{\left \{
x_{0}\right \}  }}\Vert b_{2}\Vert_{LC_{p_{2},\lambda_{2},P}^{\left \{
x_{0}\right \}  }}\int \limits_{2kr}^{\infty}\left(  1+\ln \frac{t}{r}\right)
^{2}\frac{\Vert f\Vert_{L_{p}(E(x_{0},t))}}{t^{\frac{\gamma}{q_{1}}%
-\gamma \left(  \lambda_{1}+\lambda_{2}\right)  -\gamma \left(  \frac{1}{p_{1}%
}+\frac{1}{p_{2}}\right)  +1}}dt,
\]
and%
\[
B_{43}\lesssim \Vert b_{1}\Vert_{LC_{p_{1},\lambda_{1},P}^{\left \{
x_{0}\right \}  }}\Vert b_{2}\Vert_{LC_{p_{2},\lambda_{2},P}^{\left \{
x_{0}\right \}  }}\int \limits_{2kr}^{\infty}\left(  1+\ln \frac{t}{r}\right)
^{2}\frac{\Vert f\Vert_{L_{p}(E(x_{0},t))}}{t^{\frac{\gamma}{q_{1}}%
-\gamma \left(  \lambda_{1}+\lambda_{2}\right)  -\gamma \left(  \frac{1}{p_{1}%
}+\frac{1}{p_{2}}\right)  +1}}dt.
\]
Finally, to estimate $B_{44}$, similar to the estimate of (\ref{11*}) and from
(\ref{b}) and (\ref{c}), we have%
\[
B_{44}\lesssim \Vert b_{1}\Vert_{LC_{p_{1},\lambda_{1},P}^{\left \{
x_{0}\right \}  }}\Vert b_{2}\Vert_{LC_{p_{2},\lambda_{2},P}^{\left \{
x_{0}\right \}  }}\int \limits_{2kr}^{\infty}\left(  1+\ln \frac{t}{r}\right)
^{2}\frac{\Vert f\Vert_{L_{p}(E(x_{0},t))}}{t^{\frac{\gamma}{q_{1}}%
-\gamma \left(  \lambda_{1}+\lambda_{2}\right)  -\gamma \left(  \frac{1}{p_{1}%
}+\frac{1}{p_{2}}\right)  +1}}dt.
\]
By the estimates of $B_{4j}$ above, where $j=1$, $2$, $3$, we know that%
\begin{align*}
\left \vert T_{\Omega,\alpha}^{P}\left(  \left(  b_{1}-\left(  b_{1}\right)
_{E}\right)  \left(  b_{2}-\left(  b_{2}\right)  _{E}\right)  f_{2}\right)
\left(  x\right)  \right \vert  & \lesssim \Vert b_{1}\Vert_{LC_{p_{1}%
,\lambda_{1},P}^{\left \{  x_{0}\right \}  }}\Vert b_{2}\Vert_{LC_{p_{2}%
,\lambda_{2},P}^{\left \{  x_{0}\right \}  }}\int \limits_{2kr}^{\infty}\left(
1+\ln \frac{t}{r}\right)  ^{2}\\
& \times \frac{\Vert f\Vert_{L_{p}(E(x_{0},t))}}{t^{\frac{\gamma}{q_{1}}%
-\gamma \left(  \lambda_{1}+\lambda_{2}\right)  -\gamma \left(  \frac{1}{p_{1}%
}+\frac{1}{p_{2}}\right)  +1}}dt.
\end{align*}
Then, we have%
\begin{align*}
B_{4}  & =\left \Vert T_{\Omega,\alpha}^{P}\left(  \left(  b_{1}-\left(
b_{1}\right)  _{E}\right)  \left(  b_{2}-\left(  b_{2}\right)  _{E}\right)
f_{2}\right)  \right \Vert _{L_{q_{1}}\left(  E\right)  }\\
& \lesssim \Vert b_{1}\Vert_{LC_{p_{1},\lambda_{1},P}^{\left \{  x_{0}\right \}
}}\Vert b_{2}\Vert_{LC_{p_{2},\lambda_{2},P}^{\left \{  x_{0}\right \}  }%
}r^{\frac{\gamma}{q_{1}}}\int \limits_{2kr}^{\infty}\left(  1+\ln \frac{t}%
{r}\right)  ^{2}\frac{\Vert f\Vert_{L_{p}(E(x_{0},t))}}{t^{\frac{\gamma}%
{q_{1}}-\gamma \left(  \lambda_{1}+\lambda_{2}\right)  -\gamma \left(  \frac
{1}{p_{1}}+\frac{1}{p_{2}}\right)  +1}}dt.
\end{align*}
So, combining all the estimates for $B_{1},$ $B_{2}$, $B_{3}$, $B_{4}$, we get%
\[
B=\left \Vert [\left(  b_{1},b_{2}\right)  ,T_{\Omega,\alpha}^{P}%
]f_{2}\right \Vert _{L_{q_{1}}\left(  E\right)  }\lesssim \Vert b_{1}%
\Vert_{LC_{p_{1},\lambda_{1},P}^{\left \{  x_{0}\right \}  }}\Vert b_{2}%
\Vert_{LC_{p_{2},\lambda_{2},P}^{\left \{  x_{0}\right \}  }}r^{\frac{\gamma
}{q_{1}}}\int \limits_{2kr}^{\infty}\left(  1+\ln \frac{t}{r}\right)  ^{2}%
\frac{\Vert f\Vert_{L_{p}(E(x_{0},t))}}{t^{\frac{\gamma}{q_{1}}-\gamma \left(
\lambda_{1}+\lambda_{2}\right)  -\gamma \left(  \frac{1}{p_{1}}+\frac{1}{p_{2}%
}\right)  +1}}dt.
\]
Thus, putting estimates $A$ and $B$ together, we get the desired conclusion%
\[
\left \Vert \lbrack \left(  b_{1},b_{2}\right)  ,T_{\Omega,\alpha}%
^{P}]f\right \Vert _{L_{q_{1}}\left(  E\right)  }\lesssim \Vert b_{1}%
\Vert_{LC_{p_{1},\lambda_{1},P}^{\left \{  x_{0}\right \}  }}\Vert b_{2}%
\Vert_{LC_{p_{2},\lambda_{2},P}^{\left \{  x_{0}\right \}  }}r^{\frac{\gamma
}{q_{1}}}\int \limits_{2kr}^{\infty}\left(  1+\ln \frac{t}{r}\right)  ^{2}%
\frac{\Vert f\Vert_{L_{p}(E(x_{0},t))}}{t^{\frac{\gamma}{q_{1}}-\gamma \left(
\lambda_{1}+\lambda_{2}\right)  -\gamma \left(  \frac{1}{p_{1}}+\frac{1}{p_{2}%
}\right)  +1}}dt.
\]
For the case of $q_{1}<s$, we can also use the same method, so we omit the
details. This completes the proof of Lemma \ref{lemma3}.
\end{proof}

\section{Proofs of the main results}

\subsection{\textbf{Proof of Theorem \ref{teo4}.}}

We consider (\ref{13}) firstly. Since $f\in LM_{p,\varphi_{1},P}^{\{x_{0}\}}
$, by (\ref{5}) and it is also non-decreasing, with respect to $t$, of the
norm $\left \Vert f\right \Vert _{L_{p}\left(  E\left(  x_{0},t\right)  \right)
}$, we get%
\begin{align}
& \frac{\left \Vert f\right \Vert _{L_{p}\left(  E\left(  x_{0},t\right)
\right)  }}{\operatorname*{essinf}\limits_{0<t<\tau<\infty}\varphi_{1}%
(x_{0},\tau)\tau^{\frac{\gamma}{p}}}\leq \operatorname*{esssup}%
\limits_{0<t<\tau<\infty}\frac{\left \Vert f\right \Vert _{L_{p}\left(  E\left(
x_{0},t\right)  \right)  }}{\varphi_{1}(x_{0},\tau)\tau^{\frac{\gamma}{p}}%
}\nonumber \\
& \leq \operatorname*{esssup}\limits_{0<\tau<\infty}\frac{\left \Vert
f\right \Vert _{L_{p}\left(  E\left(  x_{0},\tau \right)  \right)  }}%
{\varphi_{1}(x_{0},\tau)\tau^{\frac{\gamma}{p}}}\leq \Vert f\Vert
_{LM_{p,\varphi,P}^{\{x_{0}\}}}.\label{10}%
\end{align}
For $s^{\prime}\leq q<\infty$, since $(\varphi_{1},\varphi_{2})$ satisfies
(\ref{47}), we have%
\begin{align}
&
{\displaystyle \int \limits_{r}^{\infty}}
\left(  1+\ln \frac{t}{r}\right)  ^{m}\frac{\left \Vert f\right \Vert
_{L_{p}\left(  E\left(  x_{0},t\right)  \right)  }}{t^{\gamma \left(  \frac
{1}{p}-%
{\displaystyle \sum \limits_{i=1}^{m}}
\lambda_{i}\right)  +1}}dt\nonumber \\
& \leq \int \limits_{r}^{\infty}\left(  1+\ln \frac{t}{r}\right)  ^{m}%
\frac{\left \Vert f\right \Vert _{L_{p}\left(  E\left(  x_{0},t\right)  \right)
}}{\operatorname*{essinf}\limits_{t<\tau<\infty}\varphi_{1}(x_{0},\tau
)\tau^{\frac{\gamma}{p}}}\frac{\operatorname*{essinf}\limits_{t<\tau<\infty
}\varphi_{1}(x_{0},\tau)\tau^{\frac{\gamma}{p}}}{t^{\gamma \left(  \frac{1}{p}-%
{\displaystyle \sum \limits_{i=1}^{m}}
\lambda_{i}\right)  +1}}dt\nonumber \\
& \leq C\Vert f\Vert_{LM_{p,\varphi,P}^{\{x_{0}\}}}\int \limits_{r}^{\infty
}\left(  1+\ln \frac{t}{r}\right)  ^{m}\frac{\operatorname*{essinf}%
\limits_{t<\tau<\infty}\varphi_{1}(x_{0},\tau)\tau^{\frac{\gamma}{p}}%
}{t^{\gamma \left(  \frac{1}{p}-%
{\displaystyle \sum \limits_{i=1}^{m}}
\lambda_{i}\right)  +1}}dt\nonumber \\
& \leq C\Vert f\Vert_{LM_{p,\varphi,P}^{\{x_{0}\}}}\varphi_{2}(x_{0}%
,r).\label{100}%
\end{align}
Then by (\ref{200}) and (\ref{100}), we get%
\begin{align*}
\left \Vert \lbrack \overrightarrow{b},T_{\Omega}^{P}]f\right \Vert
_{LM_{q,\varphi_{2},P}^{\{x_{0}\}}}  & =\sup_{r>0}\varphi_{2}\left(
x_{0},r\right)  ^{-1}|E(x_{0},r)|^{-\frac{1}{q}}\left \Vert [\overrightarrow
{b},T_{\Omega}^{P}]f\right \Vert _{L_{q}\left(  E\left(  x_{0},r\right)
\right)  }\\
& \leq C%
{\displaystyle \prod \limits_{i=1}^{m}}
\Vert \overrightarrow{b}\Vert_{LC_{p_{i},\lambda_{i},P}^{\left \{
x_{0}\right \}  }}\sup_{r>0}\varphi_{2}\left(  x_{0},r\right)  ^{-1}\\
& \times%
{\displaystyle \int \limits_{r}^{\infty}}
\left(  1+\ln \frac{t}{r}\right)  ^{m}\frac{\Vert f\Vert_{L_{p}(E(x_{0},t))}%
}{t^{\gamma \left(  \frac{1}{p}-%
{\displaystyle \sum \limits_{i=1}^{m}}
\lambda_{i}\right)  +1}}dt\\
& \leq C%
{\displaystyle \prod \limits_{i=1}^{m}}
\Vert \overrightarrow{b}\Vert_{LC_{p_{i},\lambda_{i},P}^{\left \{
x_{0}\right \}  }}\Vert f\Vert_{LM_{p,\varphi,P}^{\{x_{0}\}}}.
\end{align*}
For the case of $p<s$, we can also use the same method, so we omit the
details. Thus, we finish the proof of (\ref{13}).

We are now in a place of proving (\ref{14}) in Theorem \ref{teo4}.

\begin{remark}
The conclusion of (\ref{14}) is a direct consequence of the following Lemma
\ref{lemma100} and (\ref{13}). In order to do this, we need to define an
operator by%
\[
\lbrack \overrightarrow{b},\tilde{T}_{\left \vert \Omega \right \vert }%
^{P}]\left(  \left \vert f\right \vert \right)  (x)=\int \limits_{{\mathbb{R}%
^{n}}}%
{\displaystyle \prod \limits_{i=1}^{m}}
\left[  \left \vert b_{i}\left(  x\right)  -b_{i}\left(  y\right)  \right \vert
\right]  \frac{\left \vert \Omega(x-y)\right \vert }{{\rho}\left(  x-y\right)
^{\gamma}}\left \vert f(y)\right \vert dy,
\]
where $\Omega \in L_{s}(S^{n-1})$, $1<s\leq \infty$, is $A_{t}$-homogeneous of
degree zero in ${\mathbb{R}^{n}}$.
\end{remark}

Using the idea of proving Lemma 2 in \cite{Ding1} (see also \cite{Gurbuz6}),
we can obtain the following pointwise relation:

\begin{lemma}
\label{lemma100}Let $\Omega \in L_{s}(S^{n-1})$, $1<s\leq \infty$, be $A_{t}
$-homogeneous of degree zero. Then we have%
\[
M_{\Omega,\overrightarrow{b}}^{P}f(x)\leq \lbrack \overrightarrow{b},\tilde
{T}_{\left \vert \Omega \right \vert }^{P}]\left(  \left \vert f\right \vert
\right)  (x)\qquad \text{for }x\in{\mathbb{R}^{n}}.
\]

\end{lemma}

In fact, for any $t>0$, we have
\begin{align*}
\lbrack \overrightarrow{b},\tilde{T}_{\left \vert \Omega \right \vert }%
^{P}]\left(  \left \vert f\right \vert \right)  (x)  & \geq%
{\displaystyle \int \limits_{{\rho}\left(  x-y\right)  <t}}
{\displaystyle \prod \limits_{i=1}^{m}}
\left[  \left \vert b_{i}\left(  x\right)  -b_{i}\left(  y\right)  \right \vert
\right]  \frac{\left \vert \Omega(x-y)\right \vert }{{\rho}\left(  x-y\right)
^{\gamma}}\left \vert f(y)\right \vert dy\\
& \geq \frac{1}{t^{\gamma}}%
{\displaystyle \int \limits_{E\left(  x,t\right)  }}
{\displaystyle \prod \limits_{i=1}^{m}}
\left[  \left \vert b_{i}\left(  x\right)  -b_{i}\left(  y\right)  \right \vert
\right]  \left \vert \Omega(x-y)\right \vert \left \vert f(y)\right \vert dy.
\end{align*}
Taking the supremum for $t>0$ on the inequality above, we get%
\[
\lbrack \overrightarrow{b},\tilde{T}_{\left \vert \Omega \right \vert }%
^{P}]\left(  \left \vert f\right \vert \right)  (x)\geq M_{\Omega
,\overrightarrow{b}}^{P}f(x)\qquad \text{for }x\in{\mathbb{R}^{n}}.
\]
From the process proving (\ref{13}), it is easy to see that the conclusions of
(\ref{13}) also hold for $[\overrightarrow{b},\tilde{T}_{\left \vert
\Omega \right \vert }^{P}]$. Combining this with Lemma \ref{lemma100}, we can
immediately obtain (\ref{14}), which completes the proof.

\subsection{\textbf{Proof of Theorem \ref{teo4*}.}}

Similar to the proof of Theorem \ref{teo4}, We consider (\ref{13*}) firstly.

For $s^{\prime}\leq q<\infty$, since $(\varphi_{1},\varphi_{2})$ satisfies
(\ref{47*}) and by (\ref{10}), we have%
\begin{align}
&
{\displaystyle \int \limits_{r}^{\infty}}
\left(  1+\ln \frac{t}{r}\right)  ^{m}\frac{\left \Vert f\right \Vert
_{L_{p}\left(  E\left(  x_{0},t\right)  \right)  }}{t^{\gamma \left(  \frac
{1}{q_{1}}-\left(
{\displaystyle \sum \limits_{i=1}^{m}}
\lambda_{i}+%
{\displaystyle \sum \limits_{i=1}^{m}}
\frac{1}{p_{i}}\right)  \right)  +1}}dt\nonumber \\
& \leq \int \limits_{r}^{\infty}\left(  1+\ln \frac{t}{r}\right)  ^{m}%
\frac{\left \Vert f\right \Vert _{L_{p}\left(  E\left(  x_{0},t\right)  \right)
}}{\operatorname*{essinf}\limits_{t<\tau<\infty}\varphi_{1}(x_{0},\tau
)\tau^{\frac{\gamma}{p}}}\frac{\operatorname*{essinf}\limits_{t<\tau<\infty
}\varphi_{1}(x_{0},\tau)\tau^{\frac{\gamma}{p}}}{t^{\gamma \left(  \frac
{1}{q_{1}}-\left(
{\displaystyle \sum \limits_{i=1}^{m}}
\lambda_{i}+%
{\displaystyle \sum \limits_{i=1}^{m}}
\frac{1}{p_{i}}\right)  \right)  +1}}dt\nonumber \\
& \leq C\Vert f\Vert_{LM_{p,\varphi,P}^{\{x_{0}\}}}\int \limits_{r}^{\infty
}\left(  1+\ln \frac{t}{r}\right)  ^{m}\frac{\operatorname*{essinf}%
\limits_{t<\tau<\infty}\varphi_{1}(x_{0},\tau)\tau^{\frac{\gamma}{p}}%
}{t^{\gamma \left(  \frac{1}{q_{1}}-\left(
{\displaystyle \sum \limits_{i=1}^{m}}
\lambda_{i}+%
{\displaystyle \sum \limits_{i=1}^{m}}
\frac{1}{p_{i}}\right)  \right)  +1}}dt\nonumber \\
& \leq C\Vert f\Vert_{LM_{p,\varphi,P}^{\{x_{0}\}}}\varphi_{2}(x_{0}%
,r).\label{100*}%
\end{align}
Then by (\ref{200*}) and (\ref{100*}), we get%
\begin{align*}
\left \Vert \lbrack \overrightarrow{b},T_{\Omega,\alpha}^{P}]f\right \Vert
_{LM_{q_{1},\varphi_{2},P}^{\{x_{0}\}}}  & =\sup_{r>0}\varphi_{2}\left(
x_{0},r\right)  ^{-1}|E(x_{0},r)|^{-\frac{1}{q_{1}}}\left \Vert
[\overrightarrow{b},T_{\Omega,\alpha}^{P}]f\right \Vert _{L_{q_{1}}\left(
E\left(  x_{0},r\right)  \right)  }\\
& \leq C%
{\displaystyle \prod \limits_{i=1}^{m}}
\Vert \overrightarrow{b}\Vert_{LC_{p_{i},\lambda_{i},P}^{\left \{
x_{0}\right \}  }}\sup_{r>0}\varphi_{2}\left(  x_{0},r\right)  ^{-1}\\
& \times%
{\displaystyle \int \limits_{r}^{\infty}}
\left(  1+\ln \frac{t}{r}\right)  ^{m}\frac{\Vert f\Vert_{L_{p}(E(x_{0},t))}%
}{t^{\gamma \left(  \frac{1}{q_{1}}-\left(
{\displaystyle \sum \limits_{i=1}^{m}}
\lambda_{i}+%
{\displaystyle \sum \limits_{i=1}^{m}}
\frac{1}{p_{i}}\right)  \right)  +1}}dt\\
& \leq C%
{\displaystyle \prod \limits_{i=1}^{m}}
\Vert \overrightarrow{b}\Vert_{LC_{p_{i},\lambda_{i},P}^{\left \{
x_{0}\right \}  }}\Vert f\Vert_{LM_{p,\varphi,P}^{\{x_{0}\}}}.
\end{align*}
For the case of $q_{1}<s$, we can also use the same method, so we omit the
details. Thus, we finish the proof of (\ref{13*}).

We are now in a place of proving (\ref{14*}) in Theorem \ref{teo4*}.

\begin{remark}
The conclusion of (\ref{14*}) is a direct consequence of the following Lemma
\ref{lemma100*} and (\ref{13*}). In order to do this, we need to define an
operator by%
\[
\lbrack \overrightarrow{b},\tilde{T}_{\left \vert \Omega \right \vert ,\alpha}%
^{P}]\left(  \left \vert f\right \vert \right)  (x)=\int \limits_{{\mathbb{R}%
^{n}}}%
{\displaystyle \prod \limits_{i=1}^{m}}
\left[  \left \vert b_{i}\left(  x\right)  -b_{i}\left(  y\right)  \right \vert
\right]  \frac{\left \vert \Omega(x-y)\right \vert }{{\rho}\left(  x-y\right)
^{\gamma-\alpha}}\left \vert f(y)\right \vert dy\qquad0<\alpha<\gamma,
\]
where $\Omega \in L_{s}(S^{n-1})$, $1<s\leq \infty$, is $A_{t}$-homogeneous of
degree zero in ${\mathbb{R}^{n}}$.
\end{remark}

Using the idea of proving Lemma 2 in \cite{Ding1} (see also \cite{Gurbuz6}),
we can obtain the following pointwise relation:

\begin{lemma}
\label{lemma100*}Let $0<\alpha<\gamma$ and $\Omega \in L_{s}(S^{n-1})$,
$1<s\leq \infty$, be $A_{t}$-homogeneous of degree zero. Then we have%
\[
M_{\Omega,\overrightarrow{b},\alpha}^{P}f(x)\leq \lbrack \overrightarrow
{b},\tilde{T}_{\left \vert \Omega \right \vert ,\alpha}^{P}]\left(  \left \vert
f\right \vert \right)  (x)\qquad \text{for }x\in{\mathbb{R}^{n}}.
\]

\end{lemma}

In fact, for any $t>0$, we have
\begin{align*}
\lbrack \overrightarrow{b},\tilde{T}_{\left \vert \Omega \right \vert ,\alpha}%
^{P}]\left(  \left \vert f\right \vert \right)  (x)  & \geq%
{\displaystyle \int \limits_{{\rho}\left(  x-y\right)  <t}}
{\displaystyle \prod \limits_{i=1}^{m}}
\left[  \left \vert b_{i}\left(  x\right)  -b_{i}\left(  y\right)  \right \vert
\right]  \frac{\left \vert \Omega(x-y)\right \vert }{{\rho}\left(  x-y\right)
^{\gamma-\alpha}}\left \vert f(y)\right \vert dy\\
& \geq \frac{1}{t^{\gamma-\alpha}}%
{\displaystyle \int \limits_{E\left(  x,t\right)  }}
{\displaystyle \prod \limits_{i=1}^{m}}
\left[  \left \vert b_{i}\left(  x\right)  -b_{i}\left(  y\right)  \right \vert
\right]  \left \vert \Omega(x-y)\right \vert \left \vert f(y)\right \vert dy.
\end{align*}
Taking the supremum for $t>0$ on the inequality above, we get%
\[
\lbrack \overrightarrow{b},\tilde{T}_{\left \vert \Omega \right \vert ,\alpha}%
^{P}]\left(  \left \vert f\right \vert \right)  (x)\geq M_{\Omega
,\overrightarrow{b},\alpha}^{P}f(x)\qquad \text{for }x\in{\mathbb{R}^{n}}.
\]
From the process proving (\ref{13*}), it is easy to see that the conclusions
of (\ref{13*}) also hold for $[\overrightarrow{b},\tilde{T}_{\left \vert
\Omega \right \vert ,\alpha}^{P}]$. Combining this with Lemma \ref{lemma100*},
we can immediately obtain (\ref{14*}), which completes the proof.

\end{document}